\documentclass[11pt]{amsart}
\usepackage[english]{babel}
\usepackage[margin=1.05in]{geometry}
\usepackage{amsmath,amssymb,amsthm,mathtools}
\usepackage{enumitem}
\usepackage{array}
\usepackage{booktabs}
\usepackage{microtype}
\usepackage{xcolor}
\usepackage{graphicx}
\usepackage{tikz}
\usepackage{pgfplots}
\pgfplotsset{compat=1.14}
\usetikzlibrary{arrows.meta,calc,positioning}
\usepackage[colorlinks=true,linkcolor=blue,citecolor=blue,urlcolor=blue]{hyperref}

\newtheorem{theorem}{Theorem}[section]
\newtheorem{lemma}[theorem]{Lemma}
\newtheorem{proposition}[theorem]{Proposition}
\newtheorem{corollary}[theorem]{Corollary}

\newtheorem{definition}[theorem]{Definition}
\newtheorem{remark}[theorem]{Remark}
\newtheorem{example}[theorem]{Example}

\newcommand{\NCF}[1]{[\![#1]\!]}
\newcommand{\D}{\mathcal D}
\newcommand{\W}{\mathcal W}
\newcommand{\sgn}{\operatorname{sgn}}
\newcommand{\coeff}[2]{[q^{#1}]#2}
\newcommand{\qint}[1]{[#1]_q}
\newcommand{\ord}{\operatorname{ord}_q}
\newcommand{\gold}{\mathcal D}
\newcommand{\sil}{\mathcal S}
\newcommand{\Isil}{I_{\mathrm{sil}}}

\newcommand{\C}{\mathbb{C}}
\newcommand{\Q}{\mathbb{Q}}
\newcommand{\R}{\mathbb{R}}
\newcommand{\Z}{\mathbb{Z}}
\newcommand{\PSL}{\mathrm{PSL}}

\tikzset{
  treenode/.style={circle,draw,inner sep=1.35pt,minimum size=5pt},
  smallnode/.style={circle,draw,inner sep=0.9pt,minimum size=4pt},
  dotnode/.style={circle,fill,inner sep=1.35pt,minimum size=4pt},
  lab/.style={font=\scriptsize},
  >=Stealth
}

\pgfplotsset{
  kappagraph/.style={
    width=.94\textwidth,
    height=.50\textwidth,
    xmin=1, xmax=2,
    xtick={1,1.2,1.4,1.6,1.8,2},
    scaled ticks=false,
    xlabel={$x$},
    grid=major,
    major grid style={gray!25},
    axis line style={black},
    tick style={black},
    every axis plot/.append style={very thick,no marks},
    enlargelimits=false
  }
}

\title{\texorpdfstring{Coefficients of $q$-real numbers:\\
their combinatorial meaning and growth}
{Coefficients of $q$-real numbers: their combinatorial meaning and growth}}

\author{Pavel Etingof}
\address{
Pavel Etingof,
Department of Mathematics, 
MIT, Cambridge, MA 02139, USA
}
\email{etingof@math.mit.edu}

\author{Valentin Ovsienko}
\address{
Valentin Ovsienko,
Centre National de la Recherche Scientifique,
Laboratoire de Math\'ematiques de Reims, UMR~9008 CNRS and
Universit\'e de Reims Champagne-Ardenne,
U.F.R. Sciences Exactes et Naturelles,
Moulin de la Housse - BP 1039,
51687 Reims cedex 2,
France}
\email{valentin.ovsienko@univ-reims.fr}

\dedicatory{To Alexander Alexandrovich Kirillov on his 90th birthday}

\date{}

\begin{document}

\begin{abstract}
A $q$-deformed real number, or ``$q$-real'', was defined by Morier-Genoud and the second author.
When $x\in\R$ such that $x\geq0$, the $q$-analogue $[x]_q$ 
is a power series with integer coefficients in one formal variable~$q$.
In general a $q$-real is a formal Laurent series.
The main goal of this paper is to study the coefficients of $q$-reals as functions on~$\R$ and give 
a combinatorial interpretation of these coefficients.
This allows us to prove a conjecture studied by several authors stating
that the $q$-deformed golden ratio has the smallest radius of convergence among
the radii of the $q$-reals associated with positive real numbers.
This is a $q$-analogue of the classical Hurwitz theorem.
Our approach is combinatorial.
We prove that for every real number $x$ in the interval $(1,2)$ the absolute value of
each coefficient of the power series representing the $q$-real $[x]_q$ is dominated
by the absolute value of the corresponding coefficient of the $q$-deformed golden ratio.
The main notion is a certain collection of ordered rooted trees associated with a $q$-real.
We prove that the golden ratio corresponds to a universal class of trees.

\end{abstract}

\maketitle

\thispagestyle{empty}

%%%%%%%%%%%%%%%%
\section{Introduction}
%%%%%%%%%%%%%%%%

The theory of $q$-deformed rational and real numbers was initiated in \cite{MGO-rationals,MGO-reals}.
For a survey, see~\cite{MGOsur}. 
For a rational $x$, the $q$-deformation $\left[x\right]_q$ is a rational function with integer coefficients.
For an arbitrary real number $x$, one obtains a formal Laurent series $\left[x\right]_q$ with integer coefficients; 
for nonnegative real numbers this is a power series.  
When $x\geq1$, the $q$-deformation of~$x$ is defined as a power series in~$q$:
\begin{equation}
\label{TayEq}
\left[x\right]_q=1+\varkappa_1q+\varkappa_2q^2+\varkappa_3q^3+\cdots
\end{equation}
with coefficients $\varkappa_{k}\in\Z$.
The recurrences
\begin{equation}
\label{SLEq}
\left[x+1\right]_q=q\left[x\right]_q+1,
\qquad\qquad
\left[-\frac{1}{x}\right]_q=-\frac{1}{q\left[x\right]_{q}}
\end{equation}
express modular, or $\PSL(2,\Z)$-equivariance.
These recurrences allow one to extend~\eqref{TayEq} to arbitrary values of $x\in\R$.

Every coefficient $\varkappa_N$ of the series~\eqref{TayEq}
can be viewed as a function in~$x$.
The properties of these functions contain the full information about the 
theory of $q$-real numbers.
The goal of this paper is to study these functions and give their combinatorial interpretation.
Let us mention that the idea to study $\varkappa_N$ as a function on $\R$ was suggested to us by Sergei Fomin.

The following polynomials
\begin{equation}
\label{qInt}
[n]_{q}:=1+q+q^{2}+\cdots+q^{n-1}=\textstyle\frac{1-q^n}{1-q}
\end{equation}
 were considered by Euler~\cite{Eul}.
They are commonly considered as a $q$-analogue of (positive) integers,
and used in the context of combinatorics and the theory of $q$-series.
 Consequently, Gauss~\cite{Gau} introduced polynomials based on $[n]_{q}$, called
 $q$-binomials.
 The polynomials $[n]_{q}$ are extensively used in combinatorics and algebra,
 and more recently in analysis and mathematical physics,
 including quantum groups and quantum calculus.
 Note also that the second expression $[n]_q=(1-q^n)/(1-q)$
 extends the notion of $q$-integers to $n<0$ as a Laurent polynomial.

It is clear that $q$-deformed integers~\eqref{qInt} satisfy the first recurrence in~\eqref{SLEq},
and are characterized by this recurrence and one value $[0]_q=0$.
The second recurrence in~\eqref{SLEq} allows one to extend the $q$-deformation from integers to rationals.
For every $x\in\Q\cup\{\infty\}$, where the additional element $\infty$ is represented by the quotient~$\frac{1}{0}$,
the $q$-deformation $\left[x\right]_q$ can be understood as a map
\begin{equation}
\label{QMap}
[\,.\,]:\Q\cup\{\infty\}\to\Z(q)
\end{equation}
with values in the space $\Z(q)$ of rational functions
with integer coefficients.
This map is uniquely determined by~\eqref{SLEq} and the image of one point, namely
$$
\left[0\right]_q:=0.
$$
This statement is implicit in~\cite{MGO-rationals};
it was explicitly proved in~\cite{LMG}.
For details, see the survey \cite[Theorem 5]{MGOsur}.
The series~\eqref{TayEq} for irrational $x$ is determined by the convergence
of Taylor series called the ``stabilization phenomenon'' in~\cite{MGO-reals}.
 
The simplest example of a $q$-irrational is the $q$-deformation of the golden ratio,
$\varphi=\frac{1+\sqrt{5}}{2}$ that starts as follows:
\begin{equation}
\label{Gold}
\begin{array}{rcl}
\left[\varphi\right]_q&=&
1 + q^2 - q^3 + 2 q^4 - 4 q^5 + 8 q^6 - 17 q^7 + 37 q^8 - 82 q^9 + 185 q^{10} \\[6pt]
&&- 423 q^{11} + 978 q^{12}-2283q^{13}+ 5373q^{14}-12735q^{15}+30372q^{16}+ \cdots
\end{array}
\end{equation}
The sequence of coefficients in~$\left[\varphi\right]_q$ coincides (up to the alternating sign and a shift)
with the remarkable OEIS sequence A004148~\cite{OEISA004148}, called the {\it generalized Catalan numbers}; see~\cite{SW}.
It is known to have many combinatorial interpretations, including peakless Motzkin paths or RNA secondary structures.
For more details, see~\cite{Pedon}.

It was shown in~\cite{LMOV} that the radius of convergence of the series $\left[\varphi\right]_q$  is equal to
$$
R_\varphi=\frac{3-\sqrt{5}}{2}\approx0.38.
$$
Based on weaker statements and computer simulations, 
it was conjectured in~\cite{LMOV}  (see also~\cite{MGOV})
that $R_\varphi$ is the smallest possible
radius of convergence for any series~\eqref{TayEq} representing a $q$-real.
This statement was understood in~\cite{LMOV} as a $q$-analogue of Hurwitz's theorem.
Recall that Hurwitz's theorem asserts that every irrational $\alpha$
has infinitely many rational approximants $p/q$ satisfying
$|\alpha-p/q|<1/(\sqrt5\,q^2)$.  The constant $1/\sqrt5$ is sharp, as witnessed
by $\varphi$ (and, equivalently, by its $\PSL(2,\Z)$-orbit); see~\cite{Cassels}.
The current status of this conjecture is the following.

\begin{enumerate}
\item
The conjecture was proved  for rational $x$ in~\cite{EGMS} using the theory of Kleinian groups;
it was also proved for quadratic irrationals in~\cite{Ren,Ren2}.
\item
For an irrational $x$, it was proved in~\cite{Etingof-qreal} that the series
$\left[x\right]_q$ has a radius of convergence greater than or equal to $2-\sqrt{3}\approx0.27$.
\end{enumerate}

Besides the radius of convergence, the (potentially larger) domain of holomorphy of $\left[x\right]_q$
was studied in~\cite{Etingof-qreal}.
For every $x\in\R$, the corresponding $q$-real $\left[x\right]_q$ is a holomorphic function
in $q\in\C$ in a certain domain of $\C$ which contains 
the real interval $(0,1)$.
Moreover, the value of $\left[x\right]_q$ for real $q\in(0,1)$ is
not simply a power series, but a well-defined real number; see~\cite{Etingof-qreal}.
Furthermore, when $q\to1$, this real value of $\left[x\right]_q$ converges to~$x$,
justifying the status of ``$q$-analogue''.

In the present paper, we prove the general conjecture of~\cite{LMOV}.
%Our main result is the following statement.

\begin{theorem}
\label{RadThm}
 For every real $x>0$ the radius of convergence of the series~$\left[x\right]_q$
is greater than or equal to $\frac{3-\sqrt{5}}{2}$.
\end{theorem}

Theorem~\ref{RadThm} belongs to analysis.
However, it will be deduced from a statement that has a combinatorial nature.
We will prove the following ``coefficientwise extremality'', or ``dominance'' property.
For every $N$, the coefficient $\varkappa_N$, viewed as a function in~$x$,
attains its maximum (up to the sign) at $x=\varphi$,
provided $x$ belongs to the interval $(1,2)$, although the maximizer need not be unique.
More precisely, the moduli of the coefficients of the $q$-deformed golden ratio 
are the greatest possible coefficients
among all $q$-reals in the interval $(1,2)$.
Let us use the standard notation $\coeff{N}{[x]_q}$ 
for the $N$th coefficient $\varkappa_N$ of the series~$\left[x\right]_q$.

The main result of this paper is as follows.

\begin{theorem}
\label{DomThm}
For every real number $x$ such that $1< x<2$,
and every $N\geq2$,
\[
  \left|\coeff{N}{[x]_q}\right|
\leq
\left|\coeff{N}{[\varphi]_q}\right|.
\]
\end{theorem}

Note that the cases $N=0,1$ are trivial as $\varkappa_0=1$ and $\varkappa_1=0$ for all~$x$.

Note that the interval \((1,2)\) is the natural normalization for this raw comparison: 
outside it one can first remove the explicit translation piece using
$$
[x+r]_q=q^r[x]_q+\qint{r},
$$ 
for an integer \(r\ge0\),
which is the iteration of the first recurrence in~\eqref{SLEq}.
Clearly, this translation does not change the radius of convergence.
Negative translations can be obtained by solving the recurrence in the opposite direction:
$$
[x-1]_q=\frac{[x]_q-1}{q}.
$$
This is relevant in the proof of Theorem~\ref{RadThm} for \(0<x<1\).
Note also that negative translations extend~\eqref{TayEq} to negative~$x$ as Laurent series.

Theorem~\ref{DomThm} immediately implies Theorem~\ref{RadThm} 
with the help of the classical Cauchy--Hadamard theorem.
However Theorem~\ref{DomThm} is substantially stronger as
it identifies the golden series as a sharp universal majorant for every normalized \(q\)-real,
and the proof provides a potentially reusable method. 
Let us outline some immediate consequences 
and possible applications of Theorem~\ref{DomThm}.

\begin{itemize}
\item
Let \(x_j\) be the continued-fraction convergents of \(x\in(1,2)\). Then, for every \(r<R_\varphi\),
the sequence of analytic functions \([x_j]_q\) converges to \( [x]_q\)
uniformly on the closed disk \(|q|\leq r\).
The normalized $q$-reals  \( [x]_q\) with \(x\in(1,2)\) therefore form a normal family on the disk \(|q|<R_\varphi\).

\item
The coefficientwise majorization and local uniform convergence 
supply a universal analytic constraint on the continued-fraction generating functions 
for finite rational models (see~\cite{Ove,MPS,AL})
and suggest a mechanism for passing from their finite rational realizations to irrational limits.

\item
Theorem~\ref{DomThm} also suggests a possible \(q\)-analogue of the Lagrange and Markov spectra. 
Instead of assigning to \(x\) a single Diophantine approximation constant, one could study the coefficient profile
\[
\bigl(
|\varkappa_0(x)|,
|\varkappa_1(x)|,
|\varkappa_2(x)|,
\ldots
\bigr),
\]
or the asymptotic invariant
\[
\Lambda_q(x)
:=
\limsup_{N\to\infty}|\varkappa_N(x)|^{1/N}
=
\frac1{R_x},
\]
where $R_x$  is the radius of convergence of \([x]_q\).
This can be an interesting direction for further study
of extremal properties of continued fractions and Diophantine approximation.
\end{itemize}

\begin{remark} {\rm 
Define the {\it Cantor line} to be the set 
\[
 \R_{\rm C}=(\R\setminus\Q)\sqcup\{r_-,r_+:r\in\Q\}
       \sqcup\{-\infty,+\infty\}
\]
ordered naturally so that $r_-<r_+$ for every
$r\in\Q$.  With the order topology, $\R_{\rm C}$ is compact, perfect and totally
disconnected, hence homeomorphic to the Cantor set.
For a finite point $\xi\in\R_{\rm C}$, define the corresponding formal Laurent series $F_\xi(z)$ by
\[
 F_y(z)=[y]_z\quad (y\notin\Q),\qquad
 F_{r_+}(z)=[r]_z,\qquad
 F_{r_-}(z)=[r]^-_z\quad (r\in\Q),
\]
where the lower rational value is the one defined in~\cite{BBL} (cf. also \cite{Etingof-qreal}, Proposition 4.6).  For
$q_0\in(0,1)$, let $F_\xi(q_0)$ denote the corresponding real $q$-value
considered in~\cite{Etingof-qreal}, which agrees with the displayed series
whenever the latter converges.  By~\cite[Propositions~4.6 and~5.1]{Etingof-qreal}
(see also~\cite{BBL}), the map $\xi\mapsto F_\xi(q_0)$ is continuous and
strictly increasing.

Let $K=[1_+,2_-]$, the order closure of $(1,2)$ in $\R_{\rm C}$.  On $K$ the
series $F_\xi$ are power series; write
$F_\xi(z)=\sum_{N\ge0}a_N(\xi)z^N$.  Stabilization, together with the two
one-sided values at each rational point, implies that every coefficient
function $a_N$ is continuous on $K$.  Theorem~\ref{DomThm}, extended to the
endpoints by continuity, gives
\[
 |a_N(\xi)|\le |\varkappa_N(\varphi)|,
 \qquad \xi\in K,
\]
for every $N$.  Put $\mathbb D(R):=
\{z\in\C:|z|<R\}$, and for an open set $U\subset \Bbb C$, denote by $\mathcal O(U)$ 
the space of holomorphic functions on $U$. Let 
\[
 \mathcal H_\varphi:=\bigl\{f(z)=\sum_{N\ge0}c_Nz^N
 \in\mathcal O(\mathbb D(R_\varphi)): |c_N|\le |\varkappa_N(\varphi)|\ \text{for all }N\bigr\}.
\]
The golden majorant and dominated convergence imply that
\[
 \Phi:K\longrightarrow\mathcal H_\varphi,
 \qquad \Phi(\xi)=F_\xi,
\]
is continuous for the compact-open topology.  It is injective, since evaluation
at any fixed $q_0\in(0,R_\varphi)$ is strictly increasing.  Since $K$ is
compact and $\mathcal H_\varphi$ is Hausdorff, $\Phi$ is a closed topological
embedding.  Combining this statement with integer translations gives analogous
closed embeddings on other bounded order intervals with nonnegative rational
endpoints, with interval-dependent coefficient bounds.
}
\end{remark}

Our proof of Theorem~\ref{DomThm} is based on a classical combinatorial idea
that goes back to Cayley; see~\cite{Cayley} and \cite{FS} for a modern exposition.
We associate a family (``forest'') of weighted rooted trees with every $q$-irrational $[x]_q$.
We then give a ``signed'' combinatorial interpretation of the coefficients of $[x]_q$:
every coefficient $\coeff{N}{[x]_q}$ counts (with coefficients $\{0,\pm1\}$) rooted trees of the corresponding family.
A specialist in enumerative combinatorics will immediately recognize the class $\D$ of trees
associated with the golden ratio as a standard Catalan-like construction. 
The class of trees that we associate to an arbitrary $q$-irrational seems to be a new idea.
The most surprising point is that one fixed positive tree class $\D$ can absorb, 
degree by degree, all the signed recursive classes arising from arbitrary real continued fractions, 
and thereby prove the extremality of the golden ratio.

Let us mention that a combinatorial interpretation of $q$-rationals has been the subject of many recent papers;
see, e.g.~\cite{Ove,MPS,AL}.
For irrationals such interpretations have been investigated in special cases~\cite{OP,HP,Pedon}.
The present paper is a first attempt to give a combinatorial interpretation of the coefficients
of the series~\eqref{TayEq} in general.
Let us also mention a probabilistic interpretation of $q$-reals
suggested in a very recent preprint~\cite{Pro}.
It would be interesting to understand whether there is any relation 
between our combinatorial interpretation and this probabilistic interpretation.

For every irrational \(x>0\), the radius of convergence of the series \([x]_q\) is at most \(1\).
Indeed, an integer-coefficient power series with radius greater than \(1\) is a polynomial.
After integer translation, we can assume that $x\in(1,2)$. 
If $[x]_q\in\Z[q]$, then $[x]_1$ is an integer, whereas the limit
proved in~\cite{Etingof-qreal} gives $[x]_q\to x\notin\Z$ as $q\to1^-$.
Let us mention that $q$-rationals with maximal radius of convergence
which is equal to~$1$ have been studied in~\cite{EVW}.

The structure of this paper is as follows.
In Section~\ref{DefSec}, we recall very briefly the notion of $q$-deformed real numbers.
We introduce an auxiliary ``tail'' series $W_x$ associated with a $q$-real $[x]_q$ when $x\in(1,2)$
such that $[x]_q=1+q^2W_x$.
The tail series for the golden ratio provides a model series that will be compared to any other.
In Section~\ref{TreeSec}, we give an inductive construction of a family of signed rooted trees
associated with every $q$-real.
In Section~\ref{ProofSec}, we prove our main Theorem~\ref{DomThm}.
A particularly simple and aesthetically pleasant example of $\left[\sqrt{2}\right]_q$
(which is the normalized silver ratio) is
elaborated in Section~\ref{SSilver}.
We simplify our general constructions in this case.
In Section~\ref{LastSec}, we describe the coefficient functions
$\varkappa_N$ on the interval $(1,2)$, unveiling a surprising relation to the
Fibonacci sequence.  Finally, in the Appendix we present empirical observations
for a finite sample of rational values of~$x$.

%%%%%%%%%%%%%%%%%%%
%%%%%%%%%%%%%%%%%%%
\section{Explicit formulas for \texorpdfstring{$q$}{q}-deformed real numbers}\label{DefSec}
%%%%%%%%%%%%%%%%%%%
%%%%%%%%%%%%%%%%%%%

In this section, we briefly recall the notion of $q$-real numbers and give an explicit formula to calculate them.
The details and proofs are omitted and can be found in~\cite{MGO-reals,MGOsur}.

%%%%%%%%%%%%%%%%%%%
\subsection{\texorpdfstring{$q$}{q}-rationals}
%%%%%%%%%%%%%%%%%%%

Consider first the case of rationals $x\in\Q$.
Every rational has continued fraction expansion 
with minus signs
$$
x
 \quad =\quad
a_1 - \cfrac{1}{a_2
          - \cfrac{1}{\ddots - \cfrac{1}{a_k} } } ,
$$
where $a_i$ are integers such that $a_i\geq2$ (except for $a_1$ which can be an arbitrary integer).
This expansion is often called the negative, or Hirzebruch continued fraction, and is usually denoted by
$x=\NCF{a_1,a_2,\ldots,a_k}$.

An explicit formula for the $q$-deformations $[x]_q$ then reads
\begin{equation}
\label{qc}
[x]_{q} \quad =\quad
[a_1]_{q} - \cfrac{q^{a_{1}-1}}{[a_2]_{q} 
          - \cfrac{q^{a_{2}-1}}{\ddots \cfrac{\ddots}{[a_{k-1}]_{q}- \cfrac{q^{a_{k-1}-1}}{[a_k]_{q}} } }} ,
\end{equation}
where $[a]_q$ is the $q$-integer~\eqref{qInt}.
Note that the expression~\eqref{qc} is valid for $a_1<0$ in which case the
$q$-analogue $[a_1]_q$ is defined by $[a]_q=(1-q^a)/(1-q)$.
This is an alternative simple way to extend the notion of $q$-rationals for $x<0$,
consistent with the recurrence~\eqref{SLEq}.

\begin{example}
{\rm
Consider the remarkable sequence of rationals $\frac{F_{n+1}}{F_n}$
given by the quotients of consecutive Fibonacci numbers.
The corresponding negative continued fraction expansion is
$$
\frac{F_{2m+1}}{F_{2m}}=\NCF{2,\underbrace{3,\ldots,3}_{m-1}},
\qquad\qquad
\frac{F_{2m+2}}{F_{2m+1}}=
\NCF{2,\underbrace{3,\ldots,3}_{m-1},2},
$$
 for \(m\geq1\).
Applying~\eqref{qc}, one obtains a sequence of rational functions
$$
\left[\frac{5}{3}\right]_{q}=
\frac{1+q+2q^{2}+q^{3}}{1+q+q^{2}},
\qquad\qquad
\left[\frac{8}{5}\right]_q=
\frac{1+2q+2q^2+2q^3+q^4}{1+2q+q^2+q^3},
\ldots
$$
}
\end{example}

%%%%%%%%%%%%%%%%%%%
\subsection{\texorpdfstring{$q$}{q}-irrationals}
%%%%%%%%%%%%%%%%%%%

An irrational $x$ can be represented by an infinite negative continued fraction
$$
x
 \quad =\quad
a_1 - \cfrac{1}{a_2
          - \cfrac{1}{a_3 - \ddots } }
$$
and its $q$-deformation is again of the form
\begin{equation}
\label{qcInf}
[x]_{q} \quad =\quad
[a_1]_q - \cfrac{q^{a_{1}-1}}{[a_2]_q
          - \cfrac{q^{a_{2}-1}}{[a_3]_q - \ddots } }.
\end{equation}
The stabilization theorem of~\cite{MGO-reals} shows that
the coefficients of the finite truncations stabilize, and hence that~\eqref{qcInf}
defines an element of $\Z((q))$; for $x\ge0$ it belongs to $\Z[[q]]$.

\begin{example}
\label{PiExample}
{\rm
(a)
The negative continued fraction expansion of the golden ratio is
$\varphi=\NCF{2,3,3,\ldots}$.
Applying~\eqref{qcInf}, one obtains the series~\eqref{Gold}.
An alternative way to calculate~\eqref{Gold} is to consider the sequence of Taylor series of
the rational functions $\left[\frac{F_{n+1}}{F_n}\right]_q$ that stabilizes to~\eqref{Gold}.

(b)
A much more complicated example is the $q$-deformed $\pi$, also considered in~\cite{MGO-reals}.
The negative continued fraction of $\pi$ starts
$$
\pi=\NCF{4,2,2,2,2,2,2,17,294,3,4,5,16,2,3,4,2,4,2,3,\ldots}
$$
One then obtains
$$
\begin{array}{rcl}
\left[\pi\right]_q&=&
1+q+q^2+q^{10}-q^{12}-q^{13}+q^{15}+q^{16} - q^{20}-2q^{21}- q^{22}+2q^{23}+4q^{24}+q^{25}\\[4pt]
&&-4q^{27}-4q^{28}-2q^{29}+q^{30}+5q^{31}+8q^{32}+3q^{33}-3q^{34}-10q^{35}-12q^{36}-5q^{37}\\[4pt]
&&+8q^{38}+19q^{39}+20q^{40}+2q^{41}-18q^{42}-32q^{43}-25q^{44}+31q^{46}+51q^{47}+45q^{48}\\[4pt]
&&-7q^{49}-65q^{50}- 94q^{51}- 57q^{52}+ 35q^{53}+122q^{54}+ 140q^{55} + 72q^{56}- 76q^{57}+\cdots
\end{array}
$$
Coefficients of this series grow much slower than those of $\left[\varphi\right]_q$, and no pattern
for the coefficients of~$\left[\pi\right]_q$ is known.
}
\end{example}

%%%%%%%%%%%%%%%%%%%
\subsection{Comparison with regular continued fractions}\label{CompSec}
%%%%%%%%%%%%%%%%%%%

There exists a similar explicit formula using more standard, regular continued fraction expansion of $x$.
Recall that the regular continued fraction is
$$
x         \quad =\quad
b_1 + \cfrac{1}{b_2 
          + \cfrac{1}{b_3 +\cfrac{1}{\cdots} } } .
$$
The standard notation is:
$x=[b_1,b_2,b_3,\ldots]$.
Once again, the expansion is finite if and only if $x$ is rational.
The formula for $[x]_q$ in terms of the regular continued fractions is
more complicated and involves inversion of the parameter~$q$; 
see, e.g.,~\cite{MGO-rationals} and~\cite{MGOsur}.

The coefficients in the two expansions for the same real number $x$:
$$
x=\NCF{a_1,a_2,a_3,\ldots}=[b_1,b_2,b_3,\ldots]
$$ 
are connected by the Hirzebruch formula:
\begin{equation}
\label{HZRegEqt}
(a_1,a_2,a_3,a_4,\ldots)=
\big(b_1+1,\underbrace{2,\ldots,2}_{b_2-1},\,
b_3+2,\underbrace{2,\ldots,2}_{b_4-1},\ldots,
b_{2m-1}+2,\underbrace{2,\ldots,2}_{b_{2m}-1}, \ldots\big).
\end{equation}
This expression can be found in~\cite[Eq. (19), p.241]{Hir}.
For a detailed proof, see~\cite{MGO}.

Observe that if $1<{}x<2$, then $a_1=2$ and therefore the digits $a_i$ in~\eqref{HZRegEqt}
consist of blocks of the form $2^{(k)}n$ defined as a string of \(k\) consecutive digits \(2\) 
followed by the digit \(n>2\).
This will be important for inductive proof of Theorem~\ref{DomThm}.

%%%%%%%%%%%%%%%%%%%
\subsection{The ``tail'' series \texorpdfstring{$W_x$}{W x}}
%%%%%%%%%%%%%%%%%%%

Consider a real $x\in(1,2)$;
the corresponding negative continued fraction is of the form
$$
x=\NCF{2,a_1,a_2,a_3,\ldots},
$$
where $a_i\geq2$.
For a rational number, we choose the standard infinite expansion which is
eventually equal to~$2$; this is obtained from a finite expansion by increasing
its last digit by~$1$ and appending $2,2,2,\ldots$.

It follows immediately from~\eqref{qcInf} that
\begin{equation}
\label{WxDef}
 [x]_q=1+q^2W_x,
\end{equation}
or equivalently 
\begin{equation}
\label{X2}
W_x:=\frac{ [x]_q-1}{q^2}.
\end{equation}
If $x\in(1,2)$ is rational, \eqref{X2} is still well-defined.
Thus, we will not distinguish between rational and irrational $x$ when working with $W_x$.

It will also be useful to have another equivalent expression for the series~$W_x$ in terms of the
``tail'' of $x$ defined by the negative continued fraction
$$
\tilde{x}=\NCF{a_1,a_2,a_3,\ldots},
$$

\begin{lemma}
\label{TaiL}
One has
\begin{equation}
\label{DefW}
W_x=\frac{[\tilde{x}]_q-1}{q[\tilde{x}]_q}.
\end{equation}
\end{lemma}
\begin{proof}
The expression~\eqref{qcInf} reads
$$
[x]_q=
1+q-\frac{q}{[\tilde{x}]_q},
$$
which readily implies~\eqref{DefW}.
\end{proof}

Conversely, one has
\begin{equation}
\label{Conv}
[\tilde{x}]_q=\frac{1}{1-qW_x}.
\end{equation}

%%%%%%%%%%%%%%%%%%%
\subsection{The tail of \texorpdfstring{$\varphi$}{varphi} and the series \texorpdfstring{$D$}{D}}
%%%%%%%%%%%%%%%%%%%

Consider the simple case of the golden ratio $\varphi$.
The series $W_\varphi=\frac{\left[\varphi\right]_q-1}{q^2}$ has coefficients alternating in sign,
and it will be convenient to get rid of this alternating sign.
To this end, we replace $q$ by $-q$ and will work with the series
\begin{equation}
\label{SerD}
D(q):=\frac{\left[\varphi\right]_{-q}-1}{q^2}.
\end{equation}

The series $D$ has a positive coefficient sequence: 
$$
\begin{array}{rcl}
D(q)&=&
1 + q + 2 q^2 + 4 q^3 + 8 q^4 +17 q^5 + 37 q^6 + 82 q^7 + 185 q^{8} \\[6pt]
&&+ 423 q^{9} + 978 q^{10}+2283q^{11}+ 5373q^{12}+12735q^{13}+30372q^{14}+ \cdots
\end{array}
$$
This is a shift of OEIS sequence A004148~\cite{OEISA004148}; more precisely,
$[q^n]D(q)=\mathrm{A004148}(n+1)$.
 Its coefficients count various combinatorial objects; see~\cite{Pedon}.  
The series \(D\) will be our universal comparison series.

Let us mention that the series \eqref{SerD}
can be written (see~\cite{OP}) as the following remarkable continued fractions:
$$
D(q)\quad=\quad
\cfrac{1}{1-
           \cfrac{q}{1-
          \cfrac{q}{1-
          \cfrac{q^3}{ 1-
          \cfrac{q}{ 1-
          \cfrac{q}{1-
          \cfrac{q^3}{\ddots}}}}}}}
\quad=\quad
\cfrac{1}{1-q
          - \cfrac{q^2}{1-q
          -\cfrac{q^{3}}{1-q
          - \cfrac{q^2}{ 1-q
          - \cfrac{q^3}{ \ddots}}}}} 
$$
that bear a striking resemblance with the classical continued fractions
for the generating functions of the Catalan and Motzkin numbers, respectively;
see~\cite{OP} and references therein.

%%%%%%%%%%%%%%%%%%%
%%%%%%%%%%%%%%%%%%%
\section{Families of rooted trees for a \texorpdfstring{$q$}{q}-real}\label{TreeSec}
%%%%%%%%%%%%%%%%%%%
%%%%%%%%%%%%%%%%%%%

In this section, we introduce our main tool, a family of rooted trees
associated with a $q$-deformed irrational number.
Every tree appears with a $\pm$ sign.
The main idea is to define the series $[x]_q$ recursively
with the help of the negative continued fraction expansion.
Recursive specifications of weighted ordered trees and their translation into generating-function equations are classical; 
see, for example,~\cite{FS}. 
Ordered trees are also closely related to two-coloured Motzkin paths through the bijection of~\cite{DS}. 
The novelty of the construction below lies in the signed tree families attached to continued-fraction digits and, 
more importantly, in their degree-preserving embeddings into a single universal positive class associated 
with the golden ratio.

%%%%%%%%%%%%%%%%%%%
\subsection{The ``golden'' class of rooted trees}
%%%%%%%%%%%%%%%%%%%

Let us first construct a class of rooted trees $\D$ which will be our universal comparison class.
This class is naturally associated with the series~\eqref{SerD}.

\begin{proposition}
\label{DProp}
The series $D(q)$ satisfies the following recurrence
\begin{equation}
\label{eq:D}
D=1+qD+q^2D+q^3D^2,
\end{equation}
and is completely determined by it.
\end{proposition}

\begin{proof}
It follows directly from the continued fraction expansion $\varphi=\NCF{2,3,3,\ldots}$
and from~\eqref{qcInf} that the series $\left[\varphi\right]_q$ satisfies the quadratic equation
$$
q\left[\varphi\right]^2_q-
\left(q^2+q-1 \right)\left[\varphi\right]_q -1 =0
$$
(see~\cite{MGO-reals}).
This easily implies~\eqref{eq:D}.
\end{proof}

\begin{definition}
{\rm
The class of ordered rooted trees $\D$ is constructed inductively from one root, $\circ$, denoted by $E$,
with the help of the three ``constructors'', 
where \(Q(T_1,T_2)\) has the ordered left and right subtrees \(T_1,T_2\); see Figure~\ref{fig:golden-constructors}.

\begin{figure}[ht]
\centering
\begin{tikzpicture}[scale=0.95]
  % P1
  \node[treenode] (p1r) at (2,0.25) {};
  \node[treenode] (p1c) at (2,-0.55) {};
  \draw (p1r)--node[right,lab]{$1$}(p1c);
  \node at (2,-1.05) {$P_1(T)$};
    \node[lab] at (2,-1.45) {degree cost $1$};
  % P2
  \node[treenode] (p2r) at (4.3,0.25) {};
  \node[treenode] (p2c) at (4.3,-0.55) {};
  \draw (p2r)--node[right,lab]{$2$}(p2c);
  \node at (4.3,-1.05) {$P_2(T)$};
    \node[lab] at (4.3,-1.45) {degree cost $2$};
  % Q
  \node[treenode] (qr) at (6.8,0.25) {};
  \node[treenode] (ql) at (6.35,-0.55) {};
  \node[treenode] (qrr) at (7.25,-0.55) {};
  \draw (qr)--node[left,lab]{$ $}(ql);
  \draw (qr)--node[right,lab]{$ $}(qrr);
  \node at (6.8,-1.05) {$Q(T_1,T_2)$};
  \node[lab] at (6.8,-1.45) {degree cost $3$};
\end{tikzpicture}
\caption{The three constructors of the golden class $\D$.}
\label{fig:golden-constructors}
\end{figure}
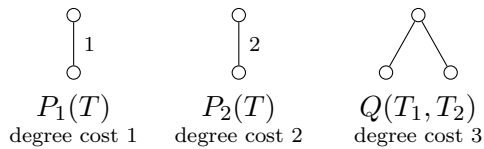

\noindent
Every tree from the class $\D$ has a well-defined degree so that 
$$
\D=\bigsqcup_{N\geq0}{\D_N}.
$$

In other words, \(\D\) is the plane rooted tree class satisfying the recurrent formula
\[
\D
=
E\sqcup P_1(\D)\sqcup P_2(\D)\sqcup Q(\D,\D),
\]
with degrees
\[
\deg E=0,
\qquad
\deg P_1(T)=1+\deg T,
\qquad
\deg P_2(T)=2+\deg T,
\]
\[
\deg Q(T_1,T_2)=3+\deg T_1+\deg T_2.
\]
}
\end{definition}

To illustrate this construction, let us give a few examples.

\begin{example}
{\rm
(a) 
The root $E$ is the only tree of degree $0$.

(b)
$P_1(E)$ is the only tree of degree $1$.

(c) 
The set $\D_2$ consists of two trees: $P_{11}(E):=P_1(P_1(E))$ and $P_2(E)$.

(d)
The set $\D_3$ of four trees of degree $3$ is presented in Figure~\ref{fig:D-degree-3}.
\begin{figure}[htbp]
\centering
\begin{tikzpicture}[x=1cm,y=1cm]

\begin{scope}[shift={(0.00,0.00)}]
  \node[treenode] (dFour1000) at (0.00,0.00) {};
  \node[treenode] (dFour1001) at (0.00,-0.6) {};
  \node[treenode] (dFour1002) at (0.00,-1.2) {};
  \node[treenode] (dFour1003) at (0.00,-1.8) {};
  \draw (dFour1002) -- node[midway,right,lab,inner sep=1pt]{$ 1 $} (dFour1003);
  \draw (dFour1001) -- node[midway,right,lab,inner sep=1pt]{$ 1 $} (dFour1002);
  \draw (dFour1000) -- node[midway,right,lab,inner sep=1pt]{$ 1 $} (dFour1001);
  \node[lab,align=center,text width=3.25cm] at (0,-2.2) {$ P_{111}(E) $};
\end{scope}
\begin{scope}[shift={(3.45,0.00)}]
  \node[treenode] (dFour2000) at (0.00,0.00) {};
  \node[treenode] (dFour2001) at (0.00,-0.6) {};
  \node[treenode] (dFour2002) at (0.00,-1.2) {};
  \draw (dFour2001) -- node[midway,right,lab,inner sep=1pt]{$ 2 $} (dFour2002);
  \draw (dFour2000) -- node[midway,right,lab,inner sep=1pt]{$ 1 $} (dFour2001);
  \node[lab,align=center,text width=3.25cm] at (0,-1.7) {$ P_{12}(E) $};
\end{scope}
\begin{scope}[shift={(6.90,0.00)}]
  \node[treenode] (dFour3000) at (0.00,0.00) {};
  \node[treenode] (dFour3001) at (0.00,-0.6) {};
  \node[treenode] (dFour3002) at (0.00,-1.2) {};
  \draw (dFour3001) -- node[midway,right,lab,inner sep=1pt]{$ 1 $} (dFour3002);
  \draw (dFour3000) -- node[midway,right,lab,inner sep=1pt]{$ 2 $} (dFour3001);
  \node[lab,align=center,text width=3.25cm] at (0,-1.7) {$ P_{21}(E) $};
\end{scope}
\begin{scope}[shift={(10.35,0.00)}]
  \node[treenode] (dFour4001) at (0.00,-0.00) {};
  \node[treenode] (dFour4002) at (-0.30,-0.7) {};
  \node[treenode] (dFour4003) at (0.30,-0.7) {};
  \draw (dFour4001) -- (dFour4002);
  \draw (dFour4001) -- (dFour4003);
  \node[lab,align=center,text width=3.25cm] at (0,-1.2) {$ Q(E,\,E) $};
\end{scope}

\end{tikzpicture}
\caption{The four trees of $\D_3=\{T\in\D:\deg T=3\}$.}
\label{fig:D-degree-3}
\end{figure}
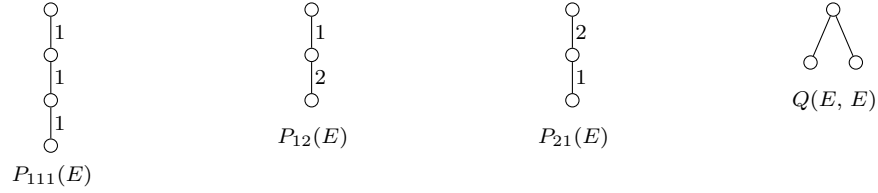

(e)
The eight trees of  $\D_4$ are depicted in Figure~\ref{fig:D-degree-4}.
Here and below we use the composition convention in symbols such as \(P_{112}(E):=P_1(P_1(P_2(E)))\).
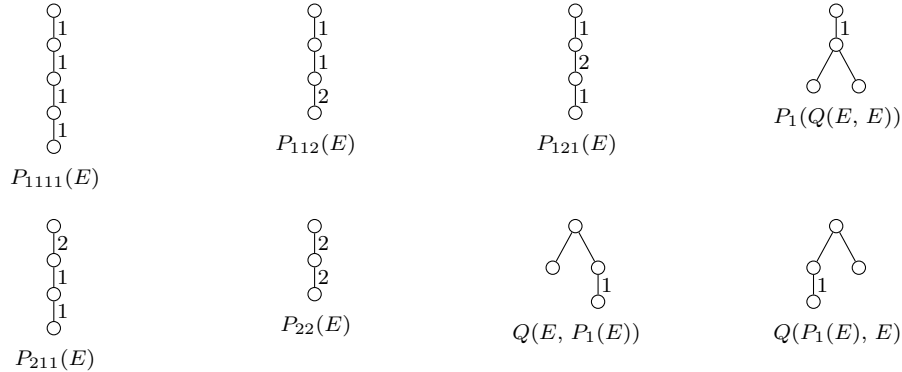
\begin{figure}[htbp]
\centering
\begin{tikzpicture}[x=1cm,y=1cm]

\begin{scope}[shift={(0.00,0.00)}]
  \node[treenode] (dFour1000) at (0.00,0.00) {};
  \node[treenode] (dFour1001) at (0.00,-0.45) {};
  \node[treenode] (dFour1002) at (0.00,-0.90) {};
  \node[treenode] (dFour1003) at (0.00,-1.35) {};
  \node[treenode] (dFour1004) at (0.00,-1.80) {};
  \draw (dFour1003) -- node[midway,right,lab,inner sep=1pt]{$ 1 $} (dFour1004);
  \draw (dFour1002) -- node[midway,right,lab,inner sep=1pt]{$ 1 $} (dFour1003);
  \draw (dFour1001) -- node[midway,right,lab,inner sep=1pt]{$ 1 $} (dFour1002);
  \draw (dFour1000) -- node[midway,right,lab,inner sep=1pt]{$ 1 $} (dFour1001);
  \node[lab,align=center,text width=3.25cm] at (0,-2.22) {$ P_{1111}(E) $};
\end{scope}
\begin{scope}[shift={(3.45,0.00)}]
  \node[treenode] (dFour2000) at (0.00,0.00) {};
  \node[treenode] (dFour2001) at (0.00,-0.45) {};
  \node[treenode] (dFour2002) at (0.00,-0.90) {};
  \node[treenode] (dFour2003) at (0.00,-1.35) {};
  \draw (dFour2002) -- node[midway,right,lab,inner sep=1pt]{$ 2 $} (dFour2003);
  \draw (dFour2001) -- node[midway,right,lab,inner sep=1pt]{$ 1 $} (dFour2002);
  \draw (dFour2000) -- node[midway,right,lab,inner sep=1pt]{$ 1 $} (dFour2001);
  \node[lab,align=center,text width=3.25cm] at (0,-1.77) {$ P_{112}(E) $};
\end{scope}
\begin{scope}[shift={(6.90,0.00)}]
  \node[treenode] (dFour3000) at (0.00,0.00) {};
  \node[treenode] (dFour3001) at (0.00,-0.45) {};
  \node[treenode] (dFour3002) at (0.00,-0.90) {};
  \node[treenode] (dFour3003) at (0.00,-1.35) {};
  \draw (dFour3002) -- node[midway,right,lab,inner sep=1pt]{$ 1 $} (dFour3003);
  \draw (dFour3001) -- node[midway,right,lab,inner sep=1pt]{$ 2 $} (dFour3002);
  \draw (dFour3000) -- node[midway,right,lab,inner sep=1pt]{$ 1 $} (dFour3001);
  \node[lab,align=center,text width=3.25cm] at (0,-1.77) {$ P_{121}(E) $};
\end{scope}
\begin{scope}[shift={(10.35,0.00)}]
  \node[treenode] (dFour4000) at (0.00,0.00) {};
  \node[treenode] (dFour4001) at (0.00,-0.45) {};
  \node[treenode] (dFour4002) at (-0.30,-1.00) {};
  \node[treenode] (dFour4003) at (0.30,-1.00) {};
  \draw (dFour4001) -- (dFour4002);
  \draw (dFour4001) -- (dFour4003);
  \draw (dFour4000) -- node[midway,right,lab,inner sep=1pt]{$ 1 $} (dFour4001);
  \node[lab,align=center,text width=3.25cm] at (0,-1.42) {$ P_{1}(Q(E,\,E)) $};
\end{scope}
\begin{scope}[shift={(0.00,-2.85)}]
  \node[treenode] (dFour5000) at (0.00,0.00) {};
  \node[treenode] (dFour5001) at (0.00,-0.45) {};
  \node[treenode] (dFour5002) at (0.00,-0.90) {};
  \node[treenode] (dFour5003) at (0.00,-1.35) {};
  \draw (dFour5002) -- node[midway,right,lab,inner sep=1pt]{$ 1 $} (dFour5003);
  \draw (dFour5001) -- node[midway,right,lab,inner sep=1pt]{$ 1 $} (dFour5002);
  \draw (dFour5000) -- node[midway,right,lab,inner sep=1pt]{$ 2 $} (dFour5001);
  \node[lab,align=center,text width=3.25cm] at (0,-1.77) {$ P_{211}(E) $};
\end{scope}
\begin{scope}[shift={(3.45,-2.85)}]
  \node[treenode] (dFour6000) at (0.00,0.00) {};
  \node[treenode] (dFour6001) at (0.00,-0.45) {};
  \node[treenode] (dFour6002) at (0.00,-0.90) {};
  \draw (dFour6001) -- node[midway,right,lab,inner sep=1pt]{$ 2 $} (dFour6002);
  \draw (dFour6000) -- node[midway,right,lab,inner sep=1pt]{$ 2 $} (dFour6001);
  \node[lab,align=center,text width=3.25cm] at (0,-1.32) {$ P_{22}(E) $};
\end{scope}
\begin{scope}[shift={(6.90,-2.85)}]
  \node[treenode] (dFour7000) at (0.00,0.00) {};
  \node[treenode] (dFour7001) at (-0.30,-0.55) {};
  \node[treenode] (dFour7002) at (0.30,-0.55) {};
  \node[treenode] (dFour7003) at (0.30,-1.00) {};
  \draw (dFour7002) -- node[midway,right,lab,inner sep=1pt]{$ 1 $} (dFour7003);
  \draw (dFour7000) -- (dFour7001);
  \draw (dFour7000) -- (dFour7002);
  \node[lab,align=center,text width=3.25cm] at (0,-1.42) {$ Q(E,\,P_{1}(E)) $};
\end{scope}
\begin{scope}[shift={(10.35,-2.85)}]
  \node[treenode] (dFour8000) at (0.00,0.00) {};
  \node[treenode] (dFour8001) at (-0.30,-0.55) {};
  \node[treenode] (dFour8002) at (-0.30,-1.00) {};
  \draw (dFour8001) -- node[midway,right,lab,inner sep=1pt]{$ 1 $} (dFour8002);
  \node[treenode] (dFour8003) at (0.30,-0.55) {};
  \draw (dFour8000) -- (dFour8001);
  \draw (dFour8000) -- (dFour8003);
  \node[lab,align=center,text width=3.25cm] at (0,-1.42) {$ Q(P_{1}(E),\,E) $};
\end{scope}

\end{tikzpicture}
\caption{The eight trees of $\D_4=\{T\in\D:\deg T=4\}$.}
\label{fig:D-degree-4}
\end{figure}

}
\end{example}

The following statement is straightforward.

\begin{proposition}
\label{GolTreeProp}
The coefficient $[q^N]D(q)$ of the series $D(q)$ counts the number of trees of degree~$N$,
that is the number of elements in $\D_N$. 
\end{proposition}

\begin{proof}
The constructors, $P_1,P_2$, and $Q$, correspond to the three terms,
$qD,q^2D$, and $q^3D^2$ in the right-hand side of~\eqref{eq:D}.
\end{proof}

%%%%%%%%%%%%%%%%%%%
\subsection{Signed rooted trees for general \texorpdfstring{$q$}{q}-irrationals}\label{TreeGen}
%%%%%%%%%%%%%%%%%%%

Let us now describe a recursive procedure that assigns a class of rooted trees
to a normalized $q$-real $[x]_q$ with $x\in(1,2)$.  Arbitrary real numbers are
reduced to this range by integer translation.
The recursion uses the following operation.

\begin{definition}
{\rm
For a positive integer \(a\ge 2\), let \(R_a\) denote the operation of prepending \(a\):
\[
  R_a(x)=\NCF{2,a,a_1,a_2,\ldots}
  \qquad\text{if}\qquad
  x=\NCF{2,a_1,a_2,\ldots}.
\]
}
\end{definition}

We will need an explicit formula for the action of $R_a$ on
the tail series $W_x$ defined by~\eqref{X2}.
Slightly abusing the notation, we will denote this action by~$R_a(W_x)$.

\begin{lemma}
\label{LittleL}
We have
\begin{equation}
\label{eq:Ra}
R_a(W_x)=
\frac{\qint{a-2}+q^{a-1}W_x}
     {\qint{a-1}+q^aW_x}.
\end{equation}
\end{lemma}

\begin{proof}
Let as above $\tilde x=\NCF{a_1,a_2,\ldots}$.
Then
$$
\NCF{a,a_1,a_2,\ldots}_q=
\qint{a}-\frac{q^{a-1}}{[\tilde x]_q}.
$$
Applying~\eqref{Conv} we have
$$
\NCF{a,a_1,a_2,\ldots}_q=
\qint{a}+q^{a}W_x-q^{a-1}=
\qint{a-1}+q^{a}W_x.
$$
(To simplify the notation, here and below we write $\NCF{a,a_1,a_2,\ldots}_q$ 
instead of \(\bigl[\NCF{a,a_1,a_2,\ldots}\bigr]_q\).)
Therefore \eqref{DefW} implies
$$
R_a(W_x)=
\frac{\qint{a-1}+q^{a}W_x-1}{q\left(\qint{a-1}+q^{a}W_x\right)}=
\frac{\qint{a-2}+q^{a-1}W_x}
     {\qint{a-1}+q^aW_x},
$$
as desired.
\end{proof}

If $H=R_a(W_x)$ and $W=W_x$, 
then~\eqref{eq:Ra} is equivalent to the recursive equation
\begin{equation}
\label{eq:single-digit-rec}
H
  =
  \qint{a-2}+q^{a-1}W
  -q\qint{a-2}H
  -q^aWH.
\end{equation}
This will be our ``tree equation'' that will allow us to associate a class of trees with $R_a(x)$,
provided $x$ already corresponds to such a class.
The inductive procedure is the following.
Each summand in~\eqref{eq:single-digit-rec} gives a constructor:
\begin{itemize}[leftmargin=2em]
  \item \(\qint{a-2}\) adds $a-2$ terminal objects of degrees $0,1,2,\ldots,a-3$;
  \item \(q^{a-1}W\) attaches a decoration of degree \(a-1\) to each previously constructed \(W\)-object
  (see Figure~\ref{fig:general-signed-constructors});
  \item \(-q\qint{a-2}H\) attaches $a-2$ unary decorations
  of degrees $1,2,3,\ldots,a-2$ to each previously constructed \(H\)-object, with negative sign;
  \item \(-q^aWH\): combine a \(W\)-object and an \(H\)-object into a new rooted tree, again with negative sign.
\end{itemize}

To make the construction more explicit, suppose that the series $W$ is already
represented by a signed graded family $\mathcal W$, and let $\mathcal H$ denote
the signed family to be constructed for $H=R_a(W)$.  Equation~\eqref{eq:single-digit-rec}
then has the following direct interpretation.  We use four types of constructors:
\begin{align*}
  &\tau_j, &&0\le j\le a-3,\\
  &U_{a-1}(T), &&T\in\mathcal W,\\
  &N_j(S), &&1\le j\le a-2,\quad S\in\mathcal H,\\
  &B_a(T,S), &&T\in\mathcal W,\quad S\in\mathcal H.
\end{align*}
Their degrees and signs are
\begin{align*}
 \deg \tau_j&=j,
 &\sgn(\tau_j)&=+1,\\
 \deg U_{a-1}(T)&=a-1+\deg T,
 &\sgn U_{a-1}(T)&=\sgn T,\\
 \deg N_j(S)&=j+\deg S,
 &\sgn N_j(S)&=-\sgn S,\\
 \deg B_a(T,S)&=a+\deg T+\deg S,
 &\sgn B_a(T,S)&=-\sgn T\,\sgn S.
\end{align*}
Thus the labels on the unary edges and at the branching vertex record the
extra degree contributed by the corresponding constructor.  Figure~\ref{fig:general-signed-constructors}
summarizes the four operations.

\begin{figure}[htbp]
\centering
\begin{tikzpicture}[x=1cm,y=1cm]
  % terminal
  \begin{scope}[shift={(0.8,0)}]
    \node[treenode] (t) at (0,0.35) {};
    \node[lab] at (0,-0.15) {$\tau_j$};
    \node[lab,align=center] at (0,-0.72) {$\deg=j$\\$\sgn=+$};
  \end{scope}

  % positive W constructor
  \begin{scope}[shift={(4.0,0)}]
    \node[treenode] (ur) at (0,0.75) {};
    \node[draw,rounded corners,minimum width=1.25cm,minimum height=0.52cm,font=\scriptsize] (ut) at (0,-0.05) {$T\in\mathcal W$};
    \draw (ur)--node[midway,right,lab]{$a-1$}(ut);
    \node[lab] at (0,-0.65) {$U_{a-1}(T)$};
    \node[lab,align=center] at (0,-1.12) {$\sgn=\sgn T$};
  \end{scope}

  % negative unary H constructor
  \begin{scope}[shift={(7.4,0)}]
    \node[treenode] (nr) at (0,0.75) {};
    \node[draw,rounded corners,minimum width=1.25cm,minimum height=0.52cm,font=\scriptsize] (ns) at (0,-0.05) {$S\in\mathcal H$};
    \draw (nr)--node[midway,right,lab]{$j$}(ns);
    \node[lab] at (0,-0.65) {$N_j(S)$};
    \node[lab,align=center] at (0,-1.12) {$1\le j\le a-2$\\$\sgn=-\sgn S$};
  \end{scope}

  % negative binary constructor
  \begin{scope}[shift={(11.2,0)}]
    \node[treenode] (br) at (0,0.75) {};
    \node[lab] at (0.28,0.72) {$a$};
    \node[draw,rounded corners,minimum width=1.15cm,minimum height=0.52cm,font=\scriptsize] (bt) at (-0.75,-0.08) {$T\in\mathcal W$};
    \node[draw,rounded corners,minimum width=1.15cm,minimum height=0.52cm,font=\scriptsize] (bs) at (0.75,-0.08) {$S\in\mathcal H$};
    \draw (br)--(bt);
    \draw (br)--(bs);
    \node[lab] at (0,-0.68) {$B_a(T,S)$};
    \node[lab,align=center] at (0,-1.12) {$\sgn=-\sgn T\,\sgn S$};
  \end{scope}
\end{tikzpicture}
\caption{The four signed constructors associated with one continued-fraction digit $a$.
The number beside an edge or a branching vertex is its degree cost.}
\label{fig:general-signed-constructors}
\end{figure}
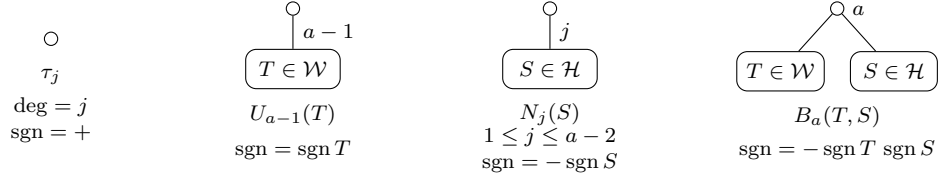

At this point the trees retain their constructor labels and signs.  In
particular, two identical unweighted shapes may still represent different
objects because their edge costs or terminal degrees are different.  The later
normal-form construction replaces this formal recursive description by a
collision-free encoding and a degree-preserving injection into the golden class $\D$.

\begin{example}
{\rm
Let
\[
 x_4=\NCF{2,4,3,3,3,\ldots}
\]
and take as input the golden tail $W=W_\varphi$.  The corresponding
series $H=W_{x_4}$ satisfies
\begin{equation}
\label{eq:digit-four-tree}
 H=1+q+q^3W-(q+q^2)H-q^4WH.
\end{equation}
The two positive terminal objects are $\tau_0$ and $\tau_1$.  Already in the
first two degrees one sees the cancellation mechanism.  In degree $1$ the
positive object $\tau_1$ cancels the negative object $N_1(\tau_0)$.  In degree
$2$ the three objects have signs $-,-,+$, as shown in
Figure~\ref{fig:digit-four-low-degrees}.

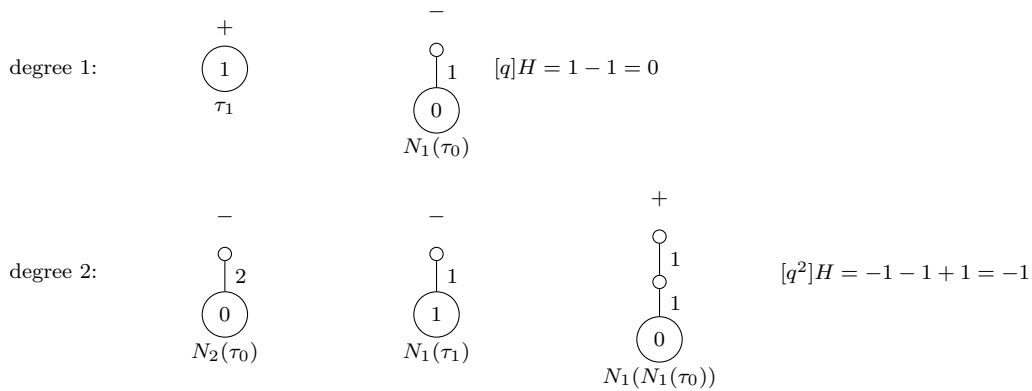
\begin{figure}[htbp]
\centering
\begin{tikzpicture}[x=1cm,y=1cm,
  terminal/.style={circle,draw,minimum size=6mm,inner sep=0pt,font=\scriptsize}]
  \node[lab,anchor=east] at (-0.2,0.25) {degree $1$:};

  \begin{scope}[shift={(1.35,0.25)}]
    \node[lab] at (0,0.55) {$+$};
    \node[terminal] (d11) at (0,0) {$1$};
    \node[lab] at (0,-0.52) {$\tau_1$};
  \end{scope}
  \begin{scope}[shift={(4.15,0.25)}]
    \node[lab] at (0,0.75) {$-$};
    \node[treenode] (d12r) at (0,0.25) {};
    \node[terminal] (d12t) at (0,-0.55) {$0$};
    \draw (d12r)--node[midway,right,lab]{$1$}(d12t);
    \node[lab] at (0,-1.05) {$N_1(\tau_0)$};
  \end{scope}
  \node[lab] at (6.0,0.25) {$[q]H=1-1=0$};

  \node[lab,anchor=east] at (-0.2,-2.45) {degree $2$:};
  \begin{scope}[shift={(1.35,-2.45)}]
    \node[lab] at (0,0.75) {$-$};
    \node[treenode] (d21r) at (0,0.25) {};
    \node[terminal] (d21t) at (0,-0.55) {$0$};
    \draw (d21r)--node[midway,right,lab]{$2$}(d21t);
    \node[lab] at (0,-1.05) {$N_2(\tau_0)$};
  \end{scope}
  \begin{scope}[shift={(4.15,-2.45)}]
    \node[lab] at (0,0.75) {$-$};
    \node[treenode] (d22r) at (0,0.25) {};
    \node[terminal] (d22t) at (0,-0.55) {$1$};
    \draw (d22r)--node[midway,right,lab]{$1$}(d22t);
    \node[lab] at (0,-1.05) {$N_1(\tau_1)$};
  \end{scope}
  \begin{scope}[shift={(7.10,-2.45)}]
    \node[lab] at (0,0.98) {$+$};
    \node[treenode] (d23r) at (0,0.48) {};
    \node[treenode] (d23m) at (0,-0.12) {};
    \node[terminal] (d23t) at (0,-0.88) {$0$};
    \draw (d23r)--node[midway,right,lab]{$1$}(d23m);
    \draw (d23m)--node[midway,right,lab]{$1$}(d23t);
    \node[lab] at (0,-1.38) {$N_1(N_1(\tau_0))$};
  \end{scope}
  \node[lab] at (10.35,-2.45) {$[q^2]H=-1-1+1=-1$};
\end{tikzpicture}
\caption{The first cancellations for the digit $a=4$.  The number inside a
terminal vertex is its intrinsic degree; edge labels are unary degree costs.}
\label{fig:digit-four-low-degrees}
\end{figure}

Using $W=1-q+2q^2-4q^3+\cdots$ in~\eqref{eq:digit-four-tree} gives
\[
 H=1-q^2+2q^3-3q^4+4q^5-6q^6+\cdots,
\]
and hence
\[
 [x_4]_q=1+q^2-q^4+2q^5-3q^6+4q^7-6q^8+\cdots.
\]
This example shows why the trees associated with a general $x$ must be signed:
different construction histories of the same total degree can cancel.
}
\end{example}

%%%%%%%%%%%%%%%%%%%
\subsection{A strategy of the proof}
%%%%%%%%%%%%%%%%%%%

The idea of our proof of Theorem \ref{DomThm} is as follows.
The desired domination for $[x]_q$ reduces to a domination statement for the coefficients of the series $W_x$
compared with the series $D$ defined in~\eqref{SerD}:
$$
  \left|\coeff{N}{W_x}\right|\le \left|\coeff{N}{D}\right|.
$$
To prove this domination, we will consider the class of trees $\D$ associated with $D$
and show that the class of trees $\W_x$ associated with $W_x$ can be embedded into $\D$.
Already this will guarantee the domination, but the sign cancellation explained above amplifies it further.

For convenience, we introduce the following general notion.

\begin{definition}
\label{def:signed-model}
{\rm
A \emph{signed golden model} for a series 
$$
A(q)=\sum_{n\ge0}a_nq^n\in\mathbb Z[[q]]
$$ 
 consists of a signed graded set \(\mathcal A\), 
with finite graded pieces, and a degree-preserving injection
\[
 I_A:\mathcal A\hookrightarrow\D
\]
such that
\[
 a_n=\sum_{X\in\mathcal A_n}\sgn(X).
\]
}
\end{definition}

The zero series has the empty signed model.  
The constant series \(1\) has the one-point model consisting of the empty word \(E\).  
If \(A\) has a signed golden model, then
\[
|a_n|\leq |\D_n|=\coeff{n}{D(q)}.
\]

The notion of signed golden model can be reformulated as follows.
For each coefficient degree \(N\ge2\), we construct a function
\[
  \varepsilon_{x,N}:\D_{N-2}\to\{0,\pm1\}
\]
such that
\begin{equation}
\label{EpsEq}
[q^N][x]_q=\sum_{T\in\D_{N-2}}\varepsilon_{x,N}(T)\quad(N\ge2).
\end{equation}
Note that the shift $N-2$ in~\eqref{EpsEq} is due to the fact that $[x]_q$ is connected to $W_x$ via
$[x]_q=1+q^2W_x$.

The domination inequality of Theorem \ref{DomThm} follows immediately 
from the existence of a signed golden model for  the tail series \(W_x\) 
associated with a given $q$-number $[x]_q$.

%%%%%%%%%%%%%%%%%%%
%%%%%%%%%%%%%%%%%%%
\section{Proof of the main theorem}\label{ProofSec}
%%%%%%%%%%%%%%%%%%%
%%%%%%%%%%%%%%%%%%%

We now prove that the operations arising from the negative continued fraction
preserve signed golden models.  
This will constitute an inductive proof of Theorem~\ref{DomThm}.

%%%%%%%%%%%%%%%%%%%
\subsection{Right-spine decomposition}
%%%%%%%%%%%%%%%%%%%

Before we start the proof, let us fix a useful notation.

For $R,T\in\D$, introduce the three right-spine operations
\[
 uR:=P_1(R),
 \qquad
 vR:=P_2(R),
 \qquad
 b(T)R:=Q(T,R).
\]
Their degree costs are respectively
\[
 1,
 \qquad 2,
 \qquad 3+\deg T.
\]
Starting from $E$, a finite sequence of symbols $u$, $v$, and $b(T)$ therefore
produces a tree in~$\D$.  We write the sequence from left to right, beginning
at the root and following the rightmost branch.  Thus, for example,
\[
 u^rvu^sE
\]
means that one first applies $r$ constructors $P_1$, then $P_2$, then $s$
constructors $P_1$ (see for instance Figure~\ref{fig:D-degree-4}).

The following lemma shows that the chosen right-spine expression is unique.

\begin{lemma}
\label{lem:right-spine}
Every tree in~$\D$ is represented by a unique finite right-spine sequence in
the symbols $u$, $v$, and $b(T)$.
\end{lemma}

\begin{proof}
Inspect the root.  A nontrivial tree is uniquely of one of the forms
$P_1(R)$, $P_2(R)$, or $Q(T,R)$.  This determines the first symbol, and in the
third case also determines the left subtree~$T$.  Repeating the argument with
the remaining right subtree gives the result.
\end{proof}

This elementary parsing will be used to embed the signed trees associated with
a general tail into~$\D$.  Unary strings $u^r$ will record the geometric-series
expansions that appear when the recursive equations are solved, while the
symbols $v$ and $b(T)$ will serve as visible delimiters.

%%%%%%%%%%%%%%%%%%
\subsection{Prepending one digit}
%%%%%%%%%%%%%%%%%%

The first closure statement treats a single digit greater than~$2$.
This is an important particular step of the proof.

\begin{lemma}
\label{lem:safe}
Let $m\ge1$.  If a series $W$ has a signed golden model, then
$R_{m+2}(W)$
has a signed golden model.
\end{lemma}

We give here two alternative proofs of this lemma.
The first is shorter and easier to read, but the second
keeps track of the construction of the family of trees presented in Section~\ref{TreeGen}
and better explains the cancellation mechanism.

\begin{proof}
Put $H=R_{m+2}(W)$.  Equation~\eqref{eq:single-digit-rec} becomes
\[
H=
\qint{m}+q^{m+1}W
-q\qint{m}H
-q^{m+2}WH.
\]
Moving the term $q\qint{m}H$ to the left gives
\begin{equation}
\label{eq:safe-divided}
H=
\frac{\qint{m}}{1+q+\cdots+q^m}
+
\frac{q^{m+1}W}{1+q+\cdots+q^m}
-
\frac{q^{m+2}W}{1+q+\cdots+q^m}\,H.
\end{equation}
We shall use the two elementary identities
\begin{align}
\frac{1}{1+q+\cdots+q^m}
 &=\sum_{s\ge0}
   \bigl(q^{s(m+1)}-q^{1+s(m+1)}\bigr),
\label{eq:unary-inverse-m}\\
\frac{\qint{m}}{1+q+\cdots+q^m}
 &=\sum_{s\ge0}
   \bigl(q^{s(m+1)}-q^{m+s(m+1)}\bigr).
\label{eq:unary-terminal-m}
\end{align}
Here and below, $u^k$ means a string of $k$ consecutive symbols $u$;
$u^0$ is the empty string.  Thus the first identity is represented by choosing,
for some $s\ge0$, either
\[
 u^{s(m+1)} \quad\hbox{with sign }+,
 \qquad\text{or}\qquad
 u^{1+s(m+1)} \quad\hbox{with sign }-.
\]
Similarly, the second identity is represented by choosing either
\[
 u^{s(m+1)} \quad\hbox{with sign }+,
 \qquad\text{or}\qquad
 u^{m+s(m+1)} \quad\hbox{with sign }-.
\]

Let $I_W:\mathcal W\hookrightarrow\D$ be a signed golden model for~$W$.
For $Y\in\mathcal W$, define
\[
 M(Y):=vu^{m-1}I_W(Y),
 \qquad
 L(Y):=b\bigl(u^{m-1}I_W(Y)\bigr).
\]
The first is a final right-spine piece of degree $m+1+\deg Y$ and begins
with the symbol $v$.  The second is the single right-spine symbol
$b(u^{m-1}I_W(Y))$; its degree cost is $m+2+\deg Y$.  Consequently,
$M(Y)$ represents the factor $q^{m+1}W$, while $L(Y)$ represents the
factor $q^{m+2}W$ in the last term of~\eqref{eq:safe-divided}.

We now construct the signed trees for~$H$.  Formula~\eqref{eq:safe-divided}
can be expanded as
\[
H=\sum_{r\ge0}
\left(
 -\frac{q^{m+2}W}{1+q+\cdots+q^m}
\right)^r
\left(
 \frac{\qint{m}}{1+q+\cdots+q^m}
 +\frac{q^{m+1}W}{1+q+\cdots+q^m}
\right),
\]
which is precisely the unique solution of~\eqref{eq:safe-divided} in
$\mathbb Z[[q]]$.

Fix $r\ge0$.  For each $i=1,\ldots,r$, choose an object
$Y_i\in\mathcal W$, choose one of the two $u$-strings representing
$1/(1+q+\cdots+q^m)$ above, and then place the symbol $L(Y_i)$.
After these $r$ blocks, finish in exactly one of the following ways:
\begin{enumerate}[label=\textup{(\roman*)},leftmargin=2.8em]
\item choose one of the two $u$-strings representing
$\qint{m}/(1+q+\cdots+q^m)$;
\item choose one of the two $u$-strings representing
$1/(1+q+\cdots+q^m)$, choose $Y_{r+1}\in\mathcal W$, and append
$M(Y_{r+1})$.
\end{enumerate}
The sign is the product of the signs of the chosen $u$-strings and of all
chosen objects $Y_i$, together with the factor $(-1)^r$.  The latter factor
comes from the minus sign in the last term of~\eqref{eq:safe-divided}.
The degrees and signs therefore agree term by term with the displayed series.
Hence the signed generating function of the constructed trees is~$H$.
For a fixed total degree, the integer $r$ and all unary lengths are bounded;
since the graded pieces of $W$ are finite, only finitely many choices of the
objects $Y_i$ occur.  Thus every graded piece of the constructed model is finite,
as required by Definition~\ref{def:signed-model}.

It remains to prove that the construction is injective.  By
Lemma~\ref{lem:right-spine}, the right-spine sequence of a tree in~$\D$ is
unique.  Every symbol of the form $b(u^{m-1}I_W(Y_i))$ marks one of the first
$r$ blocks.  Its left subtree determines $I_W(Y_i)$, and hence determines
$Y_i$ because $I_W$ is injective.  The number of consecutive $u$'s immediately
before this $b$-symbol determines which term of
\eqref{eq:unary-inverse-m} was chosen: the two possible lengths are congruent
to $0$ and $1$ modulo $m+1$, respectively.

After the last $b$-symbol, case~\textup{(i)} contains only $u$-symbols,
whereas case~\textup{(ii)} contains a first $v$-symbol, which is the beginning
of $M(Y_{r+1})$.  In case~\textup{(i)}, the two possible final lengths are
congruent to $0$ and $m$ modulo $m+1$, so they are also distinguishable.
In case~\textup{(ii)}, scan the main right spine from the root; the repeated blocks are runs of \(u\)'s 
followed by marked \(b(u^{m-1}I_W(Y_i))\)'s, while the first main-spine \(v\) marks the final case~\textup{(ii)}; 
if there is no such \(v\), one is in case~\textup{(i)}.  
Thus every constructed tree has a unique description.
The resulting map into~$\D$ is degree-preserving and injective, and therefore
gives a signed golden model for~$H$.
\end{proof}

\noindent
{\it Alternative proof.}
Let us return to the formal signed family $\mathcal H$ constructed in
Section~\ref{TreeGen}, taking $a=m+2$.  
Every object of that family has a
unique constructor chain: starting from a terminal object $\tau_j$ or
$U_{m+1}(Y)$, one repeatedly applies either a unary constructor $N_j$ or a
binary constructor $B_{m+2}(Y,\,\cdot)$.  
Read from the root toward
the terminal object, such a chain therefore splits uniquely into maximal words
of unary constructors, separated by binary constructors.  
We now check which trees in $\mathcal H$ constructed in Section~\ref{TreeGen}
cancel pairwise (because of the opposite sign);
we call them ``paired''. 
The unpaired uncancelled objects will then be encoded as actual trees in~$\D$.

Every formal constructor tree of Section~\ref{TreeGen} can be read from its root toward its terminal object. 
Consecutive unary constructors
$$
N_{j_1},N_{j_2},\ldots,N_{j_k}
$$ 
will be called a \emph{unary word} and will be denoted by
$$
(j_1,\ldots,j_k).
$$
Its degree cost is $j_1+\cdots+j_k$, and its sign is $(-1)^k$. 
A unary word is called maximal when it is bounded by 
a binary constructor or by a terminal constructor or object.
The identity \eqref{eq:unary-inverse-m} has the following direct cancellation interpretation.  
Compare $(j_1,\ldots,j_k)$, from left to right, with the
infinite alternating word
\[
  1,m,1,m,1,m,\ldots .
\]
At the first discrepancy, replace a letter $j>1$ occurring where $1$ is
expected by the pair $(1,j-1)$; conversely, replace a pair $(1,j)$ with
$j<m$, occurring where $(1,m)$ is expected, by the single letter $j+1$.
These two operations are inverse to one another, preserve the degree, and
reverse the sign.  
The unpaired unary words are precisely
\[
 (1,m)^s \quad\hbox{and}\quad (1,m)^s1,
 \qquad s\ge0,
\]
with respective degrees $s(m+1)$ and $1+s(m+1)$ and respective signs $+$ and $-$.  
We encode them by the strings $u^{s(m+1)}$ and $u^{1+s(m+1)}$.
Every other unary word
is paired with another unary word of the same degree and opposite sign, 
so the two contributions cancel.

Note that the two $u$-strings used above are not a new family
unrelated to Section~\ref{TreeGen}; they are the uncancelled normal forms of
the unary $N_j$-words.

If a final unary word ends in a terminal object $\tau_j$, first perform the
same cancellation while keeping $\tau_j$ fixed.  Among the remaining objects,
for every $1\le j\le m-1$ pair
\[
 (1,m)^s\tau_j
 \quad\longleftrightarrow\quad
 (1,m)^s1\tau_{j-1}.
\]
This again preserves the degree and reverses the sign.  The only unpaired
terminal objects are
\[
 (1,m)^s\tau_0
 \quad\hbox{and}\quad
 (1,m)^s1\tau_{m-1},
\]
which have degrees $s(m+1)$ and $m+s(m+1)$ and signs $+$ and $-$,
respectively.  Encoding them by $u^{s(m+1)}$ and
$u^{m+s(m+1)}$ gives exactly~\eqref{eq:unary-terminal-m}.

Scan the constructor chain from the root and 
apply the local replacement to the first nonnormal maximal unary word; 
if all maximal unary words are normal, apply the displayed terminal pairing in the final word. 
This gives a degree-preserving, sign-reversing involution whose fixed points are exactly the listed normal forms. 
A normalized unary word followed by
$B_{m+2}(Y,\,\cdot)$ is encoded by the corresponding $u$-string
followed by $L(Y)$; a final normalized unary word followed by
$U_{m+1}(Y)$ is encoded by the corresponding $u$-string followed by $M(Y)$;
and a chain ending in some $\tau_j$ is encoded by one of the two terminal
$u$-strings just described.  

Hence the integer $r$ appearing in the construction of the first proof 
is exactly the number of binary constructors in the normalized constructor chain. 
The normal forms obtained above are precisely the right-spine words constructed in the first proof; 
therefore the verification of degrees, signs, and injectivity is the same. 
Finally, only finitely many normal forms occur in every fixed degree, 
since every repeated block has positive degree and the graded pieces of $W$ are finite.
\qed

\begin{remark}
{\rm
(a)
The single digit \(2\) is problematic because
\[
R_2(W)=\frac{qW}{1+q^2W},
\]
and the product term starts in degree \(2\), too small for the golden binary atom \(b(T)\).  
The correct induction unit is a block
$2^{(k)}n$, where $n>2,$
that will be worked out in the next section.
This is also the natural unit in the conversion from negative to positive continued fractions; see Section~\ref{CompSec}.

(b)
The use of two different delimiters is essential but elementary.  The
nonrecursive $W$-term begins with $v$, while the recursive $WH$-term begins
with $b$.  This visible distinction is what prevents different construction
histories from producing the same golden tree.
}
\end{remark}

%%%%%%%%%%%%%%%%%%%%%%%%%%%%%%
\subsection{Continued-fraction blocks}
%%%%%%%%%%%%%%%%%%%%%%%%%%%%%%

By the Hirzebruch conversion formula~\eqref{HZRegEqt}, the natural induction
unit is a block
$2^{(k)}n,$ where $n>2.$ 
Write
\[
 p=k+1,
 \qquad
 m=n-2.
\]
Let us rewrite the equation~\eqref{eq:single-digit-rec} for one block.

\begin{lemma}
\label{lem:block-eq}
For $p\ge2$ and $m\ge1$, the series $H=R_2^{p-1}R_{m+2}(W)$ satisfies
\begin{equation}
\label{eq:block}
H=
q^{p-1}\qint{m}+q^{p+m}W
-q\qint{p}\qint{m}H
-q^{m+2}\qint{p}WH.
\end{equation}
Equivalently, if
\[
 G_m(W):=q\qint{m}+q^{m+2}W,
\]
then
\begin{equation}
\label{eq:block-short}
H=q^{p-2}G_m(W)-\qint{p}G_m(W)H.
\end{equation}
\end{lemma}

\begin{proof}
A direct induction gives
\[
R_2^{p-1}(Y)=\frac{q^{p-1}Y}{1+q^2\qint{p-1}Y}.
\]
Substitute
\[
Y=R_{m+2}(W)=
\frac{\qint{m}+q^{m+1}W}
     {1+q\qint{m}+q^{m+2}W}.
\]
A direct computation then gives~\eqref{eq:block}; equation~\eqref{eq:block-short} is the same
identity rewritten.
\end{proof}

%%%%%%%%%%%%%%%%%%%%%%%%%%
\subsection{A direct tree model for one block}
%%%%%%%%%%%%%%%%%%%%%%%%%%
The main ingredient of the proof of Theorem~\ref{DomThm} is the following statement.

\begin{lemma}
\label{thm:block-closure}
Let $p\ge2$ and $m\ge1$.  If $W$ has a signed golden model, then
$R_2^{p-1}R_{m+2}(W)$
has a signed golden model.
\end{lemma}

\begin{proof}
Let $I_W:\mathcal W\hookrightarrow\D$ be a signed golden model for~$W$.
Consider the series 
$$
G_m(W)=q\qint{m}+q^{m+2}W
$$ 
satisfying~\eqref{eq:block-short},
and separate its first term:
\[
G_m(W)=q+\Lambda_m(q),
\qquad
\Lambda_m(q)=q^2+q^3+\cdots+q^m+q^{m+2}W,
\]
where the sum $q^2+\cdots+q^m$ is empty when $m=1$.

We first describe explicitly how a term of $\Lambda_m(q)$ is inserted into a
right-spine word.  Let $\mathcal L_m(W)$ be the following signed family of
single non-$u$ right-spine entries:
\begin{itemize}[leftmargin=2em]
\item when $m\ge2$, the entry
\[
L=v,
\]
of degree $2$ and sign $+1$, corresponding to $q^2$;
\item for every $j$ with $3\le j\le m$, the entry
\[
L_j=b(u^{j-3}E),
\]
of degree $j$ and sign $+1$, corresponding to $q^j$;
\item for every $Y\in\mathcal W$, the entry
\[
L_Y=b\bigl(u^{m-1}I_W(Y)\bigr),
\]
of degree $m+2+\deg Y$ and sign $\sgn(Y)$, corresponding to
$q^{m+2}W$.
\end{itemize}
Thus the signed generating function of $\mathcal L_m(W)$ is
$\Lambda_m(q)$.

All entries in $\mathcal L_m(W)$ are distinct.  The entry $v$, when it is
present, is different from every entry beginning with $b$.  Among the fixed entries $L_j$, the left
subtree is the unary chain $u^{j-3}E$, whose length determines~$j$.
For an entry $L_Y$, the left subtree has the form
$u^{m-1}I_W(Y)$ and has degree at least $m-1$, whereas the left subtree of
every fixed entry $L_j$ has degree at most $m-3$.  Finally, two entries
$L_Y$ and $L_{Y'}$ are equal only if $Y=Y'$, by the injectivity of~$I_W$.

Substituting $G_m(W)=q+\Lambda_m(q)$ into~\eqref{eq:block-short} and moving
the term $(q+\cdots+q^p)H$ to the left gives
\begin{equation}
\label{eq:block-divided}
H=
\frac{q^{p-1}}{1+q+\cdots+q^p}
+
\frac{q^{p-2}\Lambda_m(q)}{1+q+\cdots+q^p}
-
\frac{\qint{p}\Lambda_m(q)}{1+q+\cdots+q^p}\,H.
\end{equation}
We shall use the three expansions
\begin{align}
\frac{q^{p-1}}{1+q+\cdots+q^p}
 &=\sum_{s\ge0}
 \bigl(q^{p-1+s(p+1)}-q^{p+s(p+1)}\bigr),
\label{eq:block-unary-a}\\
\frac{q^{p-2}}{1+q+\cdots+q^p}
 &=\sum_{s\ge0}
 \bigl(q^{p-2+s(p+1)}-q^{p-1+s(p+1)}\bigr),
\label{eq:block-unary-c}\\
\frac{\qint{p}}{1+q+\cdots+q^p}
 &=\sum_{s\ge0}
 \bigl(q^{s(p+1)}-q^{p+s(p+1)}\bigr).
\label{eq:block-unary-d}
\end{align}

We now define the signed trees representing~$H$.  Choose an integer $r\ge0$.
For each $i=1,\ldots,r$, choose an integer $s_i\ge0$, an entry
$L_i\in\mathcal L_m(W)$, and one of the two unary strings
\[
U_i=u^{s_i(p+1)}
\quad\hbox{with sign }+1,
\qquad\text{or}\qquad
U_i=u^{p+s_i(p+1)}
\quad\hbox{with sign }-1.
\]
Begin the right-spine word with
\[
U_1L_1U_2L_2\cdots U_rL_r.
\]
After this prefix, finish the word in exactly one of the following two ways:
\begin{enumerate}[label=\textup{(\roman*)},leftmargin=2.8em]
\item append
\[
u^{p-1+s(p+1)}E
\quad\hbox{with sign }+1,
\qquad\text{or}\qquad
u^{p+s(p+1)}E
\quad\hbox{with sign }-1,
\]
for some $s\ge0$;
\item choose one more entry $L\in\mathcal L_m(W)$ and append
\[
u^{p-2+s(p+1)}LE
\quad\hbox{with unary sign }+1,
\qquad\text{or}\qquad
u^{p-1+s(p+1)}LE
\quad\hbox{with unary sign }-1,
\]
for some $s\ge0$.  In this case the sign of the appended part is the product
of the displayed unary sign and the sign of~$L$.
\end{enumerate}
The sign of the whole object is the product of the signs of all the $U_i$,
all the $L_i$, and the final part, multiplied by $(-1)^r$.

Indeed, by~\eqref{eq:block-unary-d}, each pair $U_iL_i$, together with the
additional factor $-1$, has signed generating function
\[
-\frac{\qint{p}\Lambda_m(q)}{1+q+\cdots+q^p}.
\]
The two possible endings have signed generating functions
\[
\frac{q^{p-1}}{1+q+\cdots+q^p}
\qquad\hbox{and}\qquad
\frac{q^{p-2}\Lambda_m(q)}{1+q+\cdots+q^p},
\]
respectively, by~\eqref{eq:block-unary-a} and
\eqref{eq:block-unary-c}.  Hence the signed generating function of all the
trees just described is
\[
\sum_{r\ge0}
\left(
 -\frac{\qint{p}\Lambda_m(q)}{1+q+\cdots+q^p}
\right)^r
\left(
 \frac{q^{p-1}}{1+q+\cdots+q^p}
 +
 \frac{q^{p-2}\Lambda_m(q)}{1+q+\cdots+q^p}
\right),
\]
which is exactly the solution of~\eqref{eq:block-divided}.

It remains to check that two different choices cannot produce the same tree.
By Lemma~\ref{lem:right-spine}, the complete sequence of $u$, $v$, and
$b(T)$ entries on the right spine is uniquely recovered from the tree.
The preceding description shows how to recognize each non-$u$ entry in
$\mathcal L_m(W)$ and, in the case of $L_Y$, how to recover~$Y$: remove the
initial chain $u^{m-1}$ from its left subtree and then use the injectivity
of~$I_W$.

The form of the ending is also uniquely determined.  In case~\textup{(i)}
the right-spine word ends with a nonempty string of $u$'s, while in
case~\textup{(ii)} its last entry is an element of $\mathcal L_m(W)$.
Once the ending is known, the number of $u$'s before each non-$u$ entry is
read directly from the right-spine word.  In a repeated pair its length is
congruent to either $0$ or $p$ modulo $p+1$; in an ending of type~\textup{(i)}
it is congruent to either $p-1$ or $p$; and in an ending of type~\textup{(ii)}
it is congruent to either $p-2$ or $p-1$.  In each position the two residues
are distinct, so the choice and its sign are uniquely recovered.

Note that Definition~\ref{def:signed-model} requires finite graded pieces.  
However both this construction and the one in Lemma~\ref{lem:safe} 
involve infinitely many choices of \(r\) and of the integers \(s_i\).
We thus need to check finiteness in each degree.  
It follows because every repeated block has positive degree, 
the unary lengths are bounded in a fixed degree, and the graded pieces of \(\mathcal W\) are finite.

Thus the construction is signed, degree preserving, and injective, and
therefore gives a signed golden model for~$H$.
\end{proof}

\begin{corollary}
\label{cor:all-blocks}
For every $k\ge0$ and every $n>2$, the block operator
$R_2^kR_n$
preserves signed golden models.
\end{corollary}

\begin{proof}
For $k=0$, this is Lemma~\ref{lem:safe}.  For $k\ge1$, put
$p=k+1$ and $m=n-2$, and apply Lemma~\ref{thm:block-closure}.
\end{proof}

%%%%%%%%%%%%%%%%%%%
\subsection{The degree stability and infinite tails}
%%%%%%%%%%%%%%%%%%%

The preceding constructions apply directly to finite strings of blocks.  To
pass to an infinite continued fraction, we use the stabilization of each fixed
coefficient.

\begin{lemma}
\label{lem:lipschitz}
Let $a\ge2$ and $W,W'\in\mathbb Z[[q]]$.  Then
\begin{equation}
\label{DiffEq}
R_a(W)-R_a(W')=
\frac{q^{a-1}(W-W')}
{(1+q\qint{a-2}+q^aW)(1+q\qint{a-2}+q^aW')}.
\end{equation}
\end{lemma}

\begin{proof}
Straightforward from~\eqref{eq:Ra}. 
\end{proof}

Both denominator factors in~\eqref{DiffEq} are \(q\)-adic units.
It then follows from~\eqref{DiffEq} that
\[
\ord\bigl(R_a(W)-R_a(W')\bigr)
= \ord(W-W')+a-1
\ge \ord(W-W')+1,
\]
provided $W\not=W'$.

Thus every common leading continued-fraction digit raises the order of the
difference of two tail series by at least one.  In particular, changing the
continued fraction sufficiently far to the right cannot change any prescribed
finite set of coefficients.

\medskip
\noindent\textit{Explicit description of $\varepsilon_{x,N}$.}
The constructions above give a direct recursive way to evaluate the function
appearing in the next lemma.  For a finite string of continued-fraction blocks,
start with the empty model for the final tail $(2,2,2,\ldots)$ and prepend the
blocks from right to left.  
Starting from the empty
model, apply the signed golden models constructed in Lemmas~\ref{lem:safe} and
\ref{thm:block-closure}, parse each right-spine word according to the relevant lemma,
and define $\varepsilon_{x,N}(T)$ as the sign of the unique recursively constructed
model object mapped to $T$ (or~$0$ if none is mapped to it).

This description can be read directly from the right spine.  
Suppose that the next block is $2^{(k)}n$, with $n>2$.

If $k=0$, put $m=n-2$, as in Lemma~\ref{lem:safe}.  Before every entry
\[
 L(Y)=b\bigl(u^{m-1}I_W(Y)\bigr)
\]
the length of the preceding $u$-string must be congruent to $0$ or $1$ modulo
$m+1$; these two possibilities have signs $+$ and $-$, respectively.  The word
ends either with a pure $u$-string whose length is congruent to $0$ or $m$
modulo $m+1$, again with signs $+$ and $-$, or with
\[
 M(Y)=vu^{m-1}I_W(Y)
\]
preceded by a $u$-string of length congruent to $0$ or $1$ modulo $m+1$, with
the same two signs.  Each entry $L(Y)$ contributes the sign of $Y$ and one
additional minus sign, while $M(Y)$ contributes the sign of~$Y$.  The objects
$Y$ are then parsed recursively using the remaining continued-fraction
blocks.

If $k\ge1$, put $p=k+1$ and $m=n-2$, as in
Lemma~\ref{thm:block-closure}.  Every non-$u$ entry must belong to
$\mathcal L_m(W)$.  Before each repeated entry the length of the $u$-string is
congruent to $0$ or $p$ modulo $p+1$, with signs $+$ and $-$, respectively.
The ending is of one of the following two forms.  It either consists only of a
$u$-string of length congruent to $p-1$ or $p$ modulo $p+1$, with signs $+$ and
$-$, or it ends in an entry of $\mathcal L_m(W)$ preceded by a $u$-string of
length congruent to $p-2$ or $p-1$ modulo $p+1$, again with signs $+$ and $-$.
The sign of the tree is the product of all these unary signs and the signs of
the entries of $\mathcal L_m(W)$, together with one additional minus sign for
each repeated entry.  In an entry
\[
 L_Y=b\bigl(u^{m-1}I_W(Y)\bigr),
\]
the sign is obtained by recursively parsing~$Y$; the fixed entries $v$ and
$b(u^{j-3}E)$ have positive sign.  If at some stage the right-spine word does
not have one of the required forms, then $\varepsilon_{x,N}(T)=0$.

For a fixed degree, choose a truncation long enough that
Lemma~\ref{lem:lipschitz} makes all later digits irrelevant to the coefficient
in question, and apply the preceding parser to the resulting finite model.
This produces an admissible function $\varepsilon_{x,N}$; it may depend on the
chosen sufficiently long truncation, but its signed sum does not.  The next
lemma records the resulting existence statement.

\medskip
\noindent\textit{Example.}
Consider again
\[
 x_4=\NCF{2,4,3,3,3,\ldots}.
\]
For the input $W_\varphi$, use the signed golden model whose underlying set is
$\D$, whose sign on a tree $T$ is $(-1)^{\deg T}$, and whose injection is the
identity; this is possible because $W_\varphi(q)=D(-q)$.  The first block is the
single digit $4$, so $m=2$.  The normal-form rule gives
\[
 \varepsilon_{x_4,2}(E)=1,
 \qquad
 \varepsilon_{x_4,3}(uE)=0.
\]
The second equality is the normal-form version of the cancellation
$\tau_1-N_1(\tau_0)=0$ shown in Figure~\ref{fig:digit-four-low-degrees}.  In
degree~$2$, the only surviving golden tree is $u^2E$, with negative sign:
\[
 \varepsilon_{x_4,4}(u^2E)=-1,
 \qquad
 \varepsilon_{x_4,4}(vE)=0.
\]
In degree~$3$, the two surviving trees are $u^3E$ and $vuE$, both with positive
sign, while the other two trees do not occur:
\[
 \varepsilon_{x_4,5}(u^3E)=\varepsilon_{x_4,5}(vuE)=1,
 \qquad
 \varepsilon_{x_4,5}(uvE)=\varepsilon_{x_4,5}(b(E)E)=0.
\]
Their sums are respectively $1,0,-1,2$, in agreement with
\[
 W_{x_4}=1-q^2+2q^3-3q^4+\cdots
\]
and hence with the coefficients of $[x_4]_q=1+q^2W_{x_4}$.

\begin{lemma}
\label{thm:signed-interpretation}
For every real number $x\in(1,2)$ and every $N\ge2$, there exists a function
\[
 \varepsilon_{x,N}:\D_{N-2}\longrightarrow\{0,\pm1\}
\]
that counts the corresponding coefficient of $[x]_q$, namely
\[
 \coeff{N}{[x]_q}
 =
 \sum_{T\in\D_{N-2}}\varepsilon_{x,N}(T).
\]
\end{lemma}

\begin{proof}
Write
\[
x=\NCF{2,a_1,a_2,a_3,\ldots},
\qquad a_i\ge2.
\]

Group the tail into blocks
\[
(a_1,a_2,\ldots)
=
2^{(k_1)}n_1\,2^{(k_2)}n_2\,2^{(k_3)}n_3\cdots,
\qquad n_i>2.
\]
If this decomposition is finite, the remaining tail is $(2,2,2,\ldots)$,
whose series is $W=0$.  Starting from the empty signed model for~$0$ and
prepending the blocks from right to left, Corollary~\ref{cor:all-blocks}
gives a signed golden model for~$W_x$.

Suppose now that there are infinitely many blocks.  Keep the first $M$ blocks
and replace the rest by $(2,2,2,\ldots)$; denote the resulting tail series by
$W^{(M)}$.  Each $W^{(M)}$ has a signed golden model by the finite construction
above.  The truncated and original continued fractions have at least $M$
common leading digits, and Lemma~\ref{lem:lipschitz} therefore implies that,
for every fixed $n$, the coefficients $[q^n]W^{(M)}$ and $[q^n]W_x$ agree once
$M$ is sufficiently large.

Choose such an $M$ for $n=N-2$.  The signed golden model for $W^{(M)}$ gives a
degree-preserving injection into~$\D$.  Define $\varepsilon_{x,N}(T)$ to be
the sign of the unique model object mapped to~$T$, and $0$ when no object is
mapped to~$T$.  Since $[x]_q=1+q^2W_x$, the required identity follows.
\end{proof}

\begin{remark}
{\rm
Note that Lemma~\ref{thm:signed-interpretation} constructs the model degree by degree, 
potentially using a different truncation \(M\) for each degree.  
The constructed signed golden model attached to \(W_x\)
uses a suitable truncation separately in each degree 
and takes the disjoint union of the resulting finite graded pieces.
}
\end{remark}

\begin{proof}[Proof of Theorem~\ref{DomThm}]
The coefficients of degrees $0$ and $1$ are respectively $1$ and $0$ for every
$x\in(1,2)$.  For $N\ge2$, Lemma~\ref{thm:signed-interpretation} and
Proposition~\ref{GolTreeProp} give
\[
\left|\coeff{N}{[x]_q}\right|
\le |\D_{N-2}|
=[q^{N-2}]D(q).
\]
Since $D(q)=W_\varphi(-q)$ and
$[\varphi]_q=1+q^2W_\varphi(q)$, we have
\[
[q^{N-2}]D(q)
=
\left|\coeff{N}{[\varphi]_q}\right|.
\]
This proves the coefficientwise domination.
\end{proof}

\begin{proof}[Proof of Theorem~\ref{RadThm}]
Theorem~\ref{DomThm} and the Cauchy--Hadamard formula imply that every
$[x]_q$ with $x\in(1,2)$ has radius of convergence at least
\[
R_\varphi=\frac{3-\sqrt5}{2}.
\]
For an arbitrary $x>1$, an integer translation places $x$ in the interval
$[1,2]$.  The relation
\[
[x+r]_q=q^r[x]_q+\qint{r}
\]
changes a series only by multiplication by a monomial and addition of a
polynomial, and therefore does not change its radius of convergence.  The
integer endpoints are polynomial cases.  The result follows.

To cover the case \(0<x<1\), we use \([x]_q=\frac{[x+1]_q-1}{q}\); 
the numerator is divisible by \(q\) because \([x+1]_q\) has constant term \(1\), 
and this operation does not change the radius.
The result follows.
\end{proof}

%%%%%%%%%%%%%%%%%%%
%%%%%%%%%%%%%%%%%%%
\section{An example of the silver ratio}\label{SSilver}
%%%%%%%%%%%%%%%%%%%
%%%%%%%%%%%%%%%%%%%

In this section, we give a detailed treatment of the case of the ``silver ratio''.
We work with the normalized representative \(\sqrt2\in(1,2)\) of the silver ratio \(1+\sqrt2\), 
since integer translation does not affect the convergence radius.
Our construction simplifies the preceding constructions.
Note also that this simplified construction can be generalized for arbitrary
metallic numbers 
(a term often used for the simple quadratic irrationals
$\varphi_n=\frac{n+\sqrt{n^2+4}}{2}$; see, e.g.,~\cite{Ren,Pedon}), 
but we will not dwell on it here.

%%%%%%%%%%%%%%%%%%%
\subsection{The silver ratio}
%%%%%%%%%%%%%%%%%%%

A simple quadratic irrational is 
$$
\sqrt{2}=\NCF{2,\overline{2,4}}
 =\NCF{2,2,4,2,4,2,4,\ldots},
 $$
the corresponding $q$-deformation starts as follows:
$$
\begin{array}{rcl}
\left[\sqrt{2}\right]_q&=&
1+ q^3 - 2q^5  + q^6+ 4q^7- 5q^8- 7q^9+18q^{10} + 7q^{11}- 55q^{12}+ 18q^{13}\\[4pt]
&&+ 146q^{14} - 155q^{15}- 322q^{16}+ 692q^{17}+ 476q^{18}
 - 2446q^{19}+ 307q^{20}+\cdots
 \end{array}
 $$
The explicit generating function is
$$
\left[\sqrt{2}\right]_q=
\frac{q^3-1+\sqrt{q^6+4q^4-2q^3+4q^2+1}}{2q^2},
$$
where the square-root branch is the formal branch with constant term \(1\).
This example was studied in~\cite{MGO-reals,OP}.
Note that the coefficients of $\left[\sqrt{2}\right]_q$ form OEIS sequence A337589~\cite{OEISA337589}.

%%%%%%%%%%%%%%%%%%%
\subsection{The silver forest}
%%%%%%%%%%%%%%%%%%%

The tail series $S:=W_{\sqrt2}$ is a fixed point of the block operator
corresponding to $(2,4)$, so we have
$S=R_2R_4(S).$
Specializing~\eqref{eq:block} gives
$$
 S=q+q^2+q^4S-(q+2q^2+q^3)S-(q^4+q^5)S^2, 
$$
which can be simplified to
\begin{equation}
\label{eq:silver-fixed}
 S=q-2q^2S+q^3S-q^4S^2.
\end{equation}
Indeed, 
\[
\begin{aligned}
&S-\bigl(q+q^2+q^4S-(q+2q^2+q^3)S-(q^4+q^5)S^2\bigr)\\
&\hspace{35mm}=(1+q)\bigl(S-q+2q^2S-q^3S+q^4S^2\bigr).
\end{aligned}
\]
Since $1+q$ is a unit in $\Z[[q]]$, cancellation gives~\eqref{eq:silver-fixed}.

We now describe explicitly the signed class $\mathcal S$ of trees represented by~$S$,
according to the construction from Section~\ref{TreeGen}.
\begin{itemize}
    \item A terminal object, ``leaf'' $e$ of degree $1$ and positive sign:
    \[
    \deg(e)=1,
    \qquad
    \sgn(e)=+1.
    \]

    \item Two distinct unary constructors $U_1,U_2$, both of degree cost $2$ and both
    reversing the sign:
    \[
    \deg U_i(T)=2+\deg T,
    \qquad
    \sgn U_i(T)=-\sgn T,
    \qquad i=1,2.
    \]

    \item One unary constructor $V$ of degree cost $3$, preserving the sign:
    \[
    \deg V(T)=3+\deg T,
    \qquad
    \sgn V(T)=\sgn T.
    \]

    \item One ordered binary constructor $B$ of degree cost $4$, with an additional
    negative sign:
    \[
    \deg B(T_1,T_2)=4+\deg T_1+\deg T_2,
    \]
    \[
    \sgn B(T_1,T_2)
      =-\sgn(T_1)\sgn(T_2).
    \]
\end{itemize}

Equivalently, at the level of signed combinatorial classes,
\[
\sil
 =\{e\}
 \sqcup U_1(\sil)
 \sqcup U_2(\sil)
 \sqcup V(\sil)
 \sqcup B(\sil,\sil).
\]

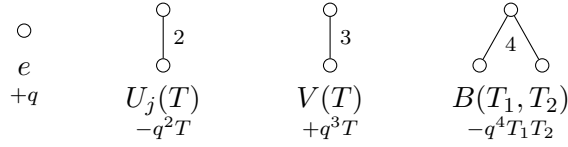
\begin{figure}[ht]
\centering
\begin{tikzpicture}[scale=0.92]
  % leaf
  \node[treenode] (l) at (0,0) {};
  \node at (0,-0.55) {$e$};
  \node[lab] at (0,-0.95) {$+q$};
  % unary negative
  \node[treenode] (ur) at (2,0.3) {};
  \node[treenode] (uc) at (2,-0.5) {};
  \draw (ur)--node[right,lab]{$2$}(uc);
  \node at (2,-1.0) {$U_j(T)$};
  \node[lab] at (2,-1.4) {$-q^2T$};
  % V positive
  \node[treenode] (vr) at (4.4,0.3) {};
  \node[treenode] (vc) at (4.4,-0.5) {};
  \draw (vr)--node[right,lab]{$3$}(vc);
  \node at (4.4,-1.0) {$V(T)$};
  \node[lab] at (4.4,-1.4) {$+q^{3}T$};
  % binary negative
  \node[treenode] (br) at (7,0.3) {};
  \node[treenode] (bl) at (6.55,-0.5) {};
  \node[treenode] (brr) at (7.45,-0.5) {};
  \draw (br)--(bl); \draw (br)--(brr);
    \node at (7,-0.2) {$\scriptstyle4$};
  \node at (7,-1.0) {$B(T_1,T_2)$};
  \node[lab] at (7,-1.4) {$-q^{4}T_1T_2$};
\end{tikzpicture}
\caption{The signed constructors for $\mathcal S$.  
The signs are the signs of the
corresponding terms in \eqref{eq:silver-fixed}.}
\label{fig:metallic-constructors}
\end{figure}

%%%%%%%%%%%%%%%%%%%
\subsection{An explicit signed degree-preserving injection into \texorpdfstring{$\D$}{D}}
%%%%%%%%%%%%%%%%%%%

Let us describe the injection of the family $\mathcal S$ into the golden family $\D$
\[
\Isil:\sil\hookrightarrow\gold.
\]
Recall that we use the right-spine notation
\[
uR:=P_1(R),
\qquad
vR:=P_2(R),
\qquad
b(T)R:=Q(T,R),
\]
for $R,T\in\gold$.

The injection is defined recursively by
\[
\begin{aligned}
\Isil(e)&=uE,\\
\Isil\bigl(U_1(T)\bigr)&=u^2\Isil(T),\\
\Isil\bigl(U_2(T)\bigr)&=v\Isil(T),\\
\Isil\bigl(V(T)\bigr)&=uv\Isil(T),\\
\Isil\bigl(B(T_1,T_2)\bigr)
   &=b\bigl(u\Isil(T_1)\bigr)\Isil(T_2).
\end{aligned}
\]
Equivalently, the last formula is
\[
\Isil\bigl(B(T_1,T_2)\bigr)
 =Q\bigl(P_1(\Isil(T_1)),\Isil(T_2)\bigr).
\]
It is easy to prove injectivity.
The initial right-spine patterns distinguish the constructors: 
\(b\) detects \(B\), \(v\) detects \(U_2\), \(uv\) detects \(V\), \(u^2\) detects \(U_1\), 
and the exceptional word \(uE\) is the leaf \(e\).

Figure~\ref{fig:silver-injection} provides simple examples of the defined injection.
\begin{figure}[ht]
\centering
\begin{tikzpicture}[scale=0.88]
  % L image
  \node at (0,0.8) {$e$};
  \draw[->] (0.4,0.8)--(1.0,0.8);
  \node[treenode] (a0) at (1.6,1.1) {};
  \node[treenode] (a1) at (1.6,0.35) {};
  \draw (a0)--node[right,lab]{$1$}(a1);
  \node at (1.6,-0.25) {$P_1(E)$};

  % U1 image
  \node at (4.0,0.8) {$U_2(e)$};
  \draw[->] (4.8,0.8)--(5.4,0.8);
  \node[treenode] (b0) at (6.0,1.45) {};
  \node[treenode] (b1) at (6.0,0.75) {};
  \node[treenode] (b2) at (6.0,0.05) {};
  \draw (b0)--node[right,lab]{$2$}(b1);
  \draw (b1)--node[right,lab]{$1$}(b2);
  \node at (6.0,-0.55) {$P_2(P_1(E))$};

  % U2 image
  \node at (0,-2.0) {$U_1(e)$};
  \draw[->] (0.9,-2.0)--(1.5,-2.0);
  \node[treenode] (c0) at (2.2,-1.25) {};
  \node[treenode] (c1) at (2.2,-1.9) {};
  \node[treenode] (c2) at (2.2,-2.55) {};
  \node[treenode] (c3) at (2.2,-3.2) {};
  \draw (c0)--node[right,lab]{$1$}(c1);
  \draw (c1)--node[right,lab]{$1$}(c2);
  \draw (c2)--node[right,lab]{$1$}(c3);
  \node at (2.2,-3.75) {$P_1(P_1(P_1(E)))$};

  % binary image
  \node at (5.0,-2.0) {$B(e,e)$};
  \draw[->] (5.7,-2.0)--(6.3,-2.0);
  \node[treenode] (d0) at (7.1,-1.4) {};
  \node[treenode] (d1) at (6.55,-2.05) {};
  \node[treenode] (d2) at (7.65,-2.05) {};
  \node[treenode] (d3) at (6.55,-2.7) {};
  \node[treenode] (d4) at (7.65,-2.7) {};
  \node[treenode] (d5) at (6.55,-3.35) {};
  \draw (d0)--(d1); \draw (d0)--(d2);
  \draw (d1)--node[left,lab]{$1$}(d3);
  \draw (d2)--node[right,lab]{$1$}(d4);
  \draw (d3)--node[right,lab]{$1$}(d5);
  \node at (7.1,-3.95) {$Q\bigl(P_1(P_1(E)),P_1(E)\bigr)$};
\end{tikzpicture}
\caption{Four examples of the silver injection $\Isil$.}
\label{fig:silver-injection}
\end{figure}
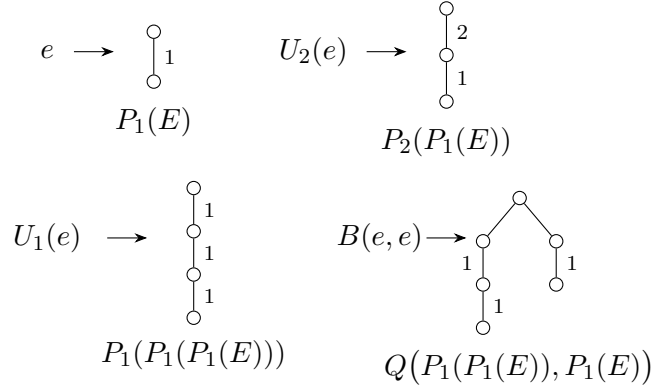

\begin{remark}
{\rm
The general block construction described in Section~\ref{ProofSec}
gives another and more complicated signed golden model for $S$, obtained from the divided
equation~\eqref{eq:block-divided}.  
This alternative model contains additional cancelling pairs and admits a normal-form
description in which unary lengths are controlled modulo $3$.  
For the silver ratio, however, the
simplified quadratic equation yields the more economical direct model described here.
}
\end{remark}

%%%%%%%%%%%%%%%%%%%
%%%%%%%%%%%%%%%%%%%
\section{Coefficient functions \texorpdfstring{$\varkappa_N$ on $(1,2)$}{kappa N on (1,2)}}
\label{LastSec}
%%%%%%%%%%%%%%%%%%%
%%%%%%%%%%%%%%%%%%%

In the final section, we will discuss the qualitative behavior of each fixed
coefficient of the series~\eqref{TayEq} as a function of~$x$.
We restrict the considerations to the interval $(1,2)$.
Note that the first recurrence in~\eqref{SLEq} reads in terms of the coefficients
$$
\varkappa_N(x+1)=\varkappa_{N-1}(x),
\qquad
N\geq1,
$$
and thus connects the coefficients for different unit intervals.

%%%%%%%%%%%%%%%%%%%
\subsection{Coefficients are step functions}
%%%%%%%%%%%%%%%%%%%
The next proposition is the first step toward showing that $\varkappa_N$ is a step function.

\begin{proposition}
\label{LoConetProp}
For every $N$, the function
\[
 x\longmapsto\varkappa_N(x)=\coeff{N}{[x]_q}
\]
is locally constant at every irrational point of $(1,2)$.
\end{proposition}

\begin{proof}
For a fixed irrational $x$, choose enough initial negative continued-fraction
digits so that Lemma~\ref{lem:lipschitz} makes all later digits irrelevant to
the coefficient of degree~$N$.  All real numbers in the corresponding
continued-fraction cylinder have the same initial digits, hence the same
coefficient~$\varkappa_N$.
\end{proof}

Let us give the first examples explicitly.

By~\eqref{WxDef}, $\varkappa_{0}(x)\equiv1$  and $\varkappa_{1}(x)\equiv0$ 
on $(1,2)$.
The coefficient $\varkappa_{2}(x)$ equals~$0$ for $x<\frac{3}{2}$ and~$1$ otherwise.
Figures~\ref{fig:kappa-3}--\ref{fig:kappa-7} illustrate the first
non-trivial coefficient functions on the interval $(1,2)$.  The vertical
segments are only a plotting convention indicating jumps; the values at the
rational breakpoints are not distinguished in these figures.
The graphs of the coefficients $\varkappa_N(x)$ become increasingly intricate
as $N$ grows.

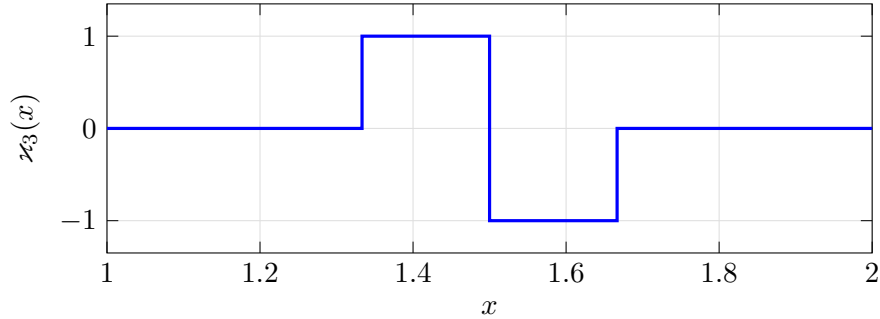
\begin{figure}[htbp]
\centering
\begin{tikzpicture}
  \begin{axis}[
    kappagraph,
    width=.72\textwidth,
    height=.30\textwidth,
    ymin=-1.35, ymax=1.35,
    ytick={-1,0,1},
    ylabel={$\varkappa_{3}(x)$}
  ]
    \addplot+[const plot] coordinates {
      (1,0)
      ({4/3},1)
      ({3/2},-1)
      ({5/3},0)
      (2,0)
    };
  \end{axis}
\end{tikzpicture}
\caption{The coefficient function $x\mapsto\varkappa_{3}(x)$ on $(1,2)$.}
\label{fig:kappa-3}
\end{figure}

\begin{figure}[htbp]
\centering
\begin{tikzpicture}
  \begin{axis}[
    kappagraph,
    width=.72\textwidth,
    height=.30\textwidth,
    ymin=-1.35, ymax=2.35,
    ytick={-1,0,1,2},
    ylabel={$\varkappa_{4}(x)$}
  ]
    \addplot+[const plot] coordinates {
      (1,0)
      ({5/4},1)
      ({4/3},-1)
      ({7/5},0)
      ({3/2},1)
      ({8/5},2)
      ({5/3},-1)
      ({7/4},0)
      (2,0)
    };
  \end{axis}
\end{tikzpicture}
\caption{The coefficient function $x\mapsto\varkappa_{4}(x)$ on $(1,2)$.}
\label{fig:kappa-4}
\end{figure}
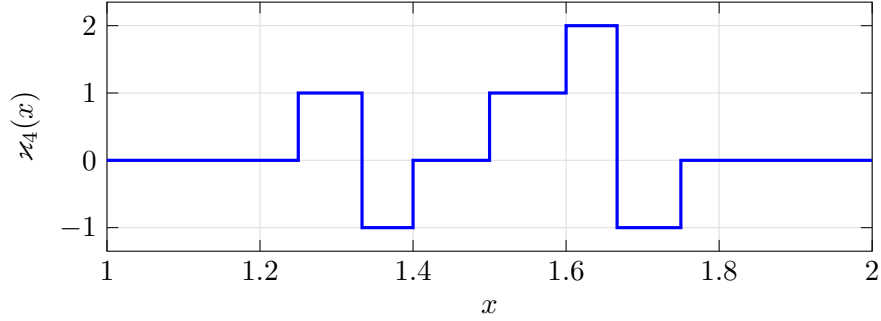

\begin{figure}[htbp]
\centering
\begin{tikzpicture}
  \begin{axis}[
    kappagraph,
    ymin=-4.45, ymax=2.45,
    ytick={-4,-3,-2,-1,0,1,2},
    ylabel={$\varkappa_{5}(x)$}
  ]
    \addplot+[const plot] coordinates {
      (1,0)
      ({6/5},1)
      ({5/4},-1)
      ({9/7},0)
      ({11/8},1)
      ({7/5},-2)
      ({10/7},-1)
      ({11/7},0)
      ({8/5},-4)
      ({13/8},-3)
      ({5/3},1)
      ({12/7},2)
      ({7/4},-1)
      ({9/5},0)
      (2,0)
    };
  \end{axis}
\end{tikzpicture}
\caption{The coefficient function $x\mapsto\varkappa_{5}(x)$ on $(1,2)$.}
\label{fig:kappa-5}
\end{figure}
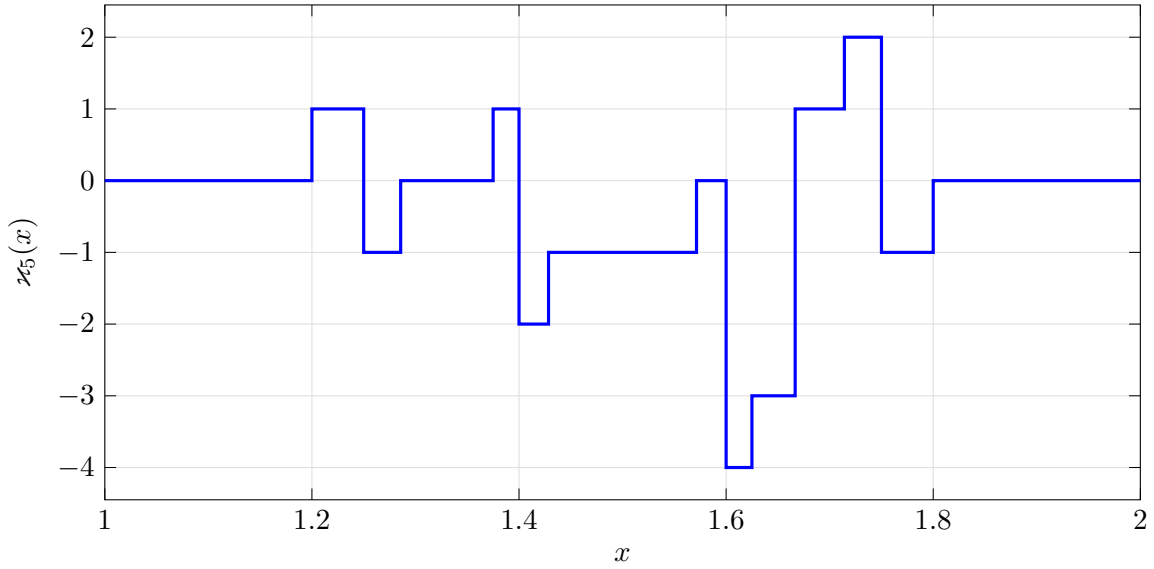

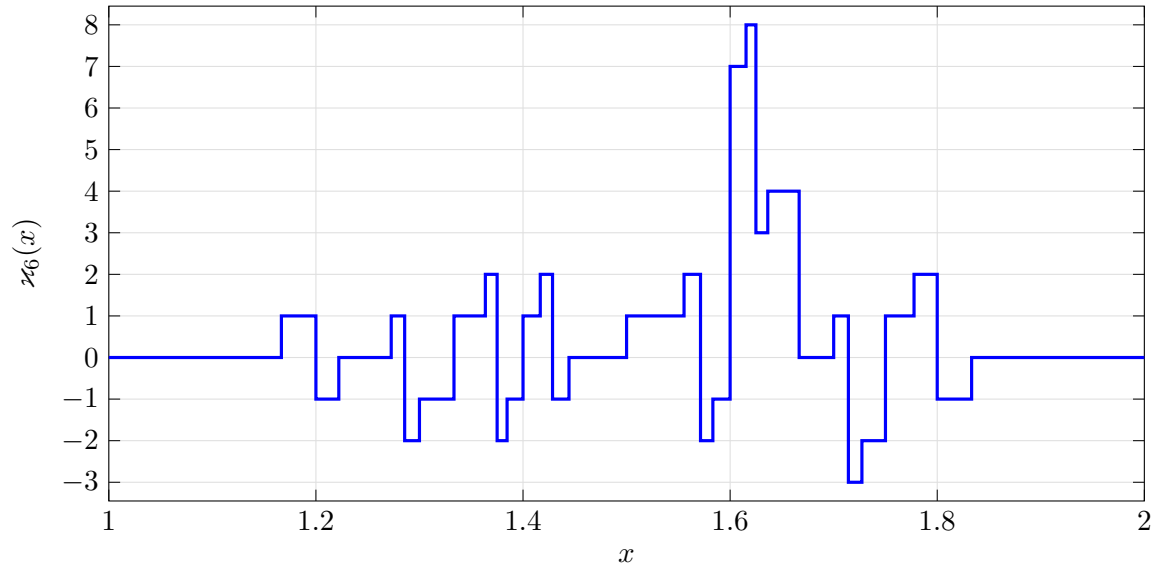
\begin{figure}[htbp]
\centering
\begin{tikzpicture}
  \begin{axis}[
    kappagraph,
    ymin=-3.45, ymax=8.45,
    ytick={-3,-2,-1,0,1,2,3,4,5,6,7,8},
    ylabel={$\varkappa_{6}(x)$}
  ]
    \addplot+[const plot] coordinates {
      (1,0)
      ({7/6},1)
      ({6/5},-1)
      ({11/9},0)
      ({14/11},1)
      ({9/7},-2)
      ({13/10},-1)
      ({4/3},1)
      ({15/11},2)
      ({11/8},-2)
      ({18/13},-1)
      ({7/5},1)
      ({17/12},2)
      ({10/7},-1)
      ({13/9},0)
      ({3/2},1)
      ({14/9},2)
      ({11/7},-2)
      ({19/12},-1)
      ({8/5},7)
      ({21/13},8)
      ({13/8},3)
      ({18/11},4)
      ({5/3},0)
      ({17/10},1)
      ({12/7},-3)
      ({19/11},-2)
      ({7/4},1)
      ({16/9},2)
      ({9/5},-1)
      ({11/6},0)
      (2,0)
    };
  \end{axis}
\end{tikzpicture}
\caption{The coefficient function $x\mapsto\varkappa_{6}(x)$ on $(1,2)$.}
\label{fig:kappa-6}
\end{figure}

\begin{figure}[htbp]
\centering
\begin{tikzpicture}
  \begin{axis}[
    kappagraph,
    ymin=-18, ymax=6,
    ytick={-15,-10,-5,0,5},
    ylabel={$\varkappa_{7}(x)$}
  ]
    \addplot+[const plot] coordinates {
      (1,0)
      ({8/7},1)
      ({7/6},-1)
      ({13/11},0)
      ({17/14},1)
      ({11/9},-2)
      ({16/13},-1)
      ({5/4},0)
      ({19/15},1)
      ({14/11},-3)
      ({23/18},-2)
      ({9/7},0)
      ({22/17},1)
      ({13/10},-2)
      ({17/13},-1)
      ({19/14},0)
      ({15/11},-4)
      ({26/19},-3)
      ({11/8},3)
      ({29/21},4)
      ({18/13},-1)
      ({25/18},0)
      ({7/5},3)
      ({24/17},4)
      ({17/12},0)
      ({27/19},1)
      ({10/7},2)
      ({23/16},3)
      ({13/9},0)
      ({16/11},1)
      ({3/2},-1)
      ({17/11},0)
      ({14/9},-4)
      ({25/16},-3)
      ({11/7},4)
      ({30/19},5)
      ({19/12},0)
      ({27/17},1)
      ({8/5},-12)
      ({29/18},-11)
      ({21/13},-17)
      ({34/21},-16)
      ({13/8},-3)
      ({31/19},-2)
      ({18/11},-7)
      ({23/14},-6)
      ({5/3},-1)
      ({22/13},0)
      ({17/10},-4)
      ({29/17},-3)
      ({12/7},3)
      ({31/18},4)
      ({19/11},-1)
      ({26/15},0)
      ({23/13},1)
      ({16/9},-3)
      ({25/14},-2)
      ({9/5},1)
      ({20/11},2)
      ({11/6},-1)
      ({13/7},0)
      (2,0)
    };
  \end{axis}
\end{tikzpicture}
\caption{The coefficient function $x\mapsto\varkappa_{7}(x)$ on $(1,2)$.}
\label{fig:kappa-7}
\end{figure}
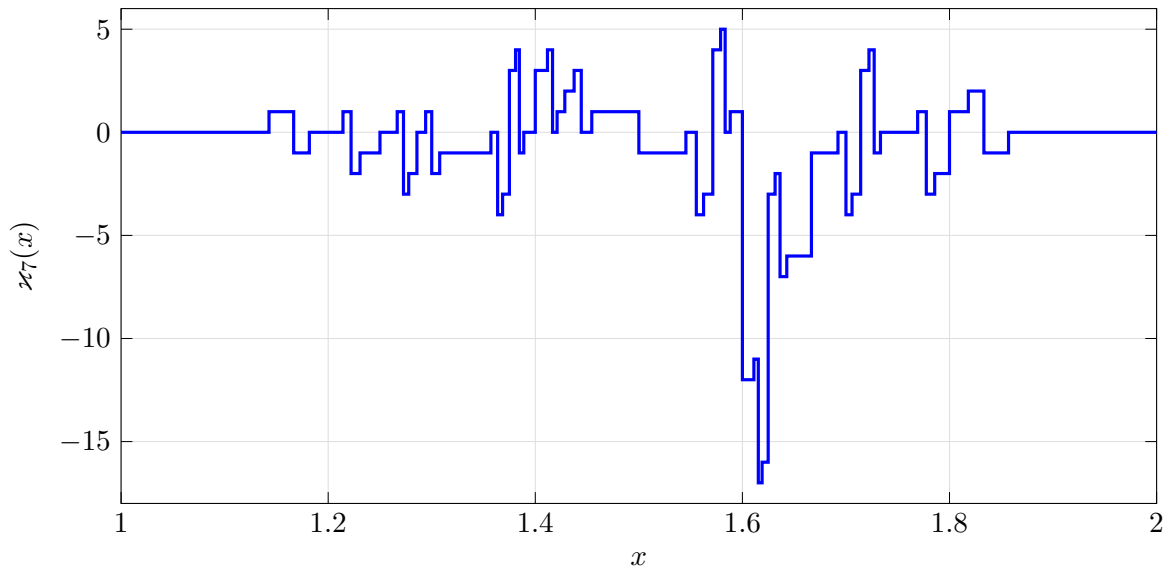

Lemma~\ref{lem:lipschitz} shows that only finitely many initial continued-fraction 
patterns of~$x$ affect~$\varkappa_N(x)$, and it is also clear that
sufficiently large continued-fraction digits are indistinguishable modulo a fixed power of $q$.
Furthermore, sufficiently large values of each relevant digit give the same series modulo \(q^{N+1}\).
It follows from~\eqref{eq:Ra} that if $a\ge N+3$, then modulo $q^{N+1}$ 
the terms $q^{a-1}W$ and $q^aW$ vanish, 
while both $[a-2]_q$ and $[a-1]_q$ are congruent to $1+q+\cdots+q^N$. 
Hence $R_a(W)\equiv1\pmod{q^{N+1}}$, independently of both $a$ and $W$. 
Together with the finite bound on the number of relevant initial digits, 
this leaves only finitely many truncated digit patterns and therefore finitely many rational endpoints.
Hence only finitely many rational breakpoints occur.
Proposition~\ref{LoConetProp} thus implies that
every coefficient $\varkappa_N$ is a step function on $(1,2)$.

We also have the following simple property.
 
\begin{proposition}
\label{RContProp}
Every coefficient function $\varkappa_N(x)$
is right-continuous at its rational breakpoints. 
\end{proposition}

\begin{proof}
The statement means that for every rational $r\in(1,2)$ and every fixed $N$, 
there exists $\varepsilon>0$ such that $\varkappa_N(x)=\varkappa_N(r)$
for $r\leq x<r+\varepsilon$.
This follows directly from~\eqref{qcInf}.
More explicitly, one can use the infinite negative expansion 
obtained by increasing the last digit and appending \(2,2,\ldots\); 
cylinders sharing a long prefix of this expansion lie on the right of the rational 
(the maps \(t\mapsto a-1/t\) are increasing).  
Lemma~\ref{lem:lipschitz} then gives the claimed right-continuity.
\end{proof}

%%%%%%%%%%%%%%%%%%
\subsection{Fibonacci maximizers}
%%%%%%%%%%%%%%%%%%

Recall that the classical Fibonacci sequence $0,1,1,2,3,5,8,13,\ldots$
is determined by $F_0=0,F_1=1$, and $F_{n+1}=F_n+F_{n-1}$.

It turns out, quite remarkably, that the maximal absolute values of the functions $\varkappa_N(x)$
are attained at the quotients of two consecutive Fibonacci numbers.
This relation of $q$-deformed real numbers to the Fibonacci sequence is quite interesting.
The next statement is closely related to Theorem~\ref{DomThm} and can be
viewed as an expected corollary of it.
It shows that the Fibonacci convergents realize the sharp coefficientwise bound degree by degree.
\begin{corollary}
\label{FiboThm}
For every $N\ge2$, put \(m=\lfloor N/2\rfloor\).
Then
$$
 \max_{1<x<2}|\varkappa_N(x)|
 =|\varkappa_N(\varphi)|
 =|\varkappa_N(r_m)|,
$$
where 
$$
 r_m=
 \frac{F_{2m+2}}{F_{2m+1}}=
 \begin{cases}
 \dfrac{F_{N+2}}{F_{N+1}},&N\text{ even},\\[8pt]
 \dfrac{F_{N+1}}{F_N},&N\text{ odd}.
 \end{cases}
$$
\end{corollary}

\begin{proof}
Theorem~\ref{DomThm} gives
$$
 |\varkappa_N(x)|\le |\varkappa_N(\varphi)|,
 \qquad 1<x<2.
$$
It therefore suffices to show that the rational number $r_m$ has the same
$N$th coefficient as $\varphi$.

In negative continued-fraction notation,
$$
 r_m=
 \NCF{2,\underbrace{3,\ldots,3}_{m\text{ times}},2,2,2,\ldots}
 =\frac{F_{2m+2}}{F_{2m+1}},
$$
whereas
$$
 \varphi=\NCF{2,3,3,3,\ldots}.
$$
Recall that $W_x$ is defined by~\eqref{WxDef}.
If, as before, $R_3$ denotes the operation of prepending the digit $3$, then 
$$
 W_{r_m}=R_3^m(0),
 \qquad
 W_\varphi=R_3^m(W_\varphi).
$$
Moreover,
\[
 R_3(U)-R_3(V)
 =\frac{q^2(U-V)}
 {(1+q+q^3U)(1+q+q^3V)}.
\]
Since both denominator factors are $q$-adic units,
\[
 \operatorname{ord}_q\bigl(R_3(U)-R_3(V)\bigr)
 =\operatorname{ord}_q(U-V)+2.
\]
As $W_\varphi$ has nonzero constant term, iteration gives
\[
 \operatorname{ord}_q(W_\varphi-W_{r_m})=2m.
\]
Hence
\[
 [\varphi]_q-[r_m]_q
 =q^2(W_\varphi-W_{r_m})
 \in q^{2m+2}\mathbb Z[[q]].
\]
Therefore
\[
 \varkappa_j(r_m)=\varkappa_j(\varphi)
 \qquad (j\le 2m+1).
\]
Since $N\le 2m+1$, the equality holds for $j=N$, and the result follows.
\end{proof}

The maximal absolute values of $ \varkappa_N(x)$ given by Corollary~\ref{FiboThm}
are precisely the values of the series~$D$:
\[
 \max_{1<x<2}|\varkappa_N(x)|=[q^{N-2}]D(q),\qquad N\ge2.
\]
Their explicit expressions can be found in~\cite{Pedon}.

%%%%%%%%%%%%%%%%%%%%%%
\section*{Appendix: some statistical results}
%%%%%%%%%%%%%%%%%%%%%%

It would be interesting to study the radius of convergence of $[x]_q$ 
for a ``typical" real number $x\in (1,2)$; for instance, we know nothing at all about the 
measure of the set of $x$ for which $[x]_q$ converges for $|q|<1$.  
In the hope to gain some intuition for this question, we computed, 
for each of the $76{,}115$ reduced rationals $x=p/q\in(1,2)$ with
$2\le q\le500$, the minimal modulus of a pole of the rational
function $[x]_q$ (i.e., the radius of convergence of its Taylor series). 
The results are depicted in the histogram in Figure 13. 

\begin{figure}[htbp]\label{f13}
\centering
\includegraphics[width=.45\textwidth]{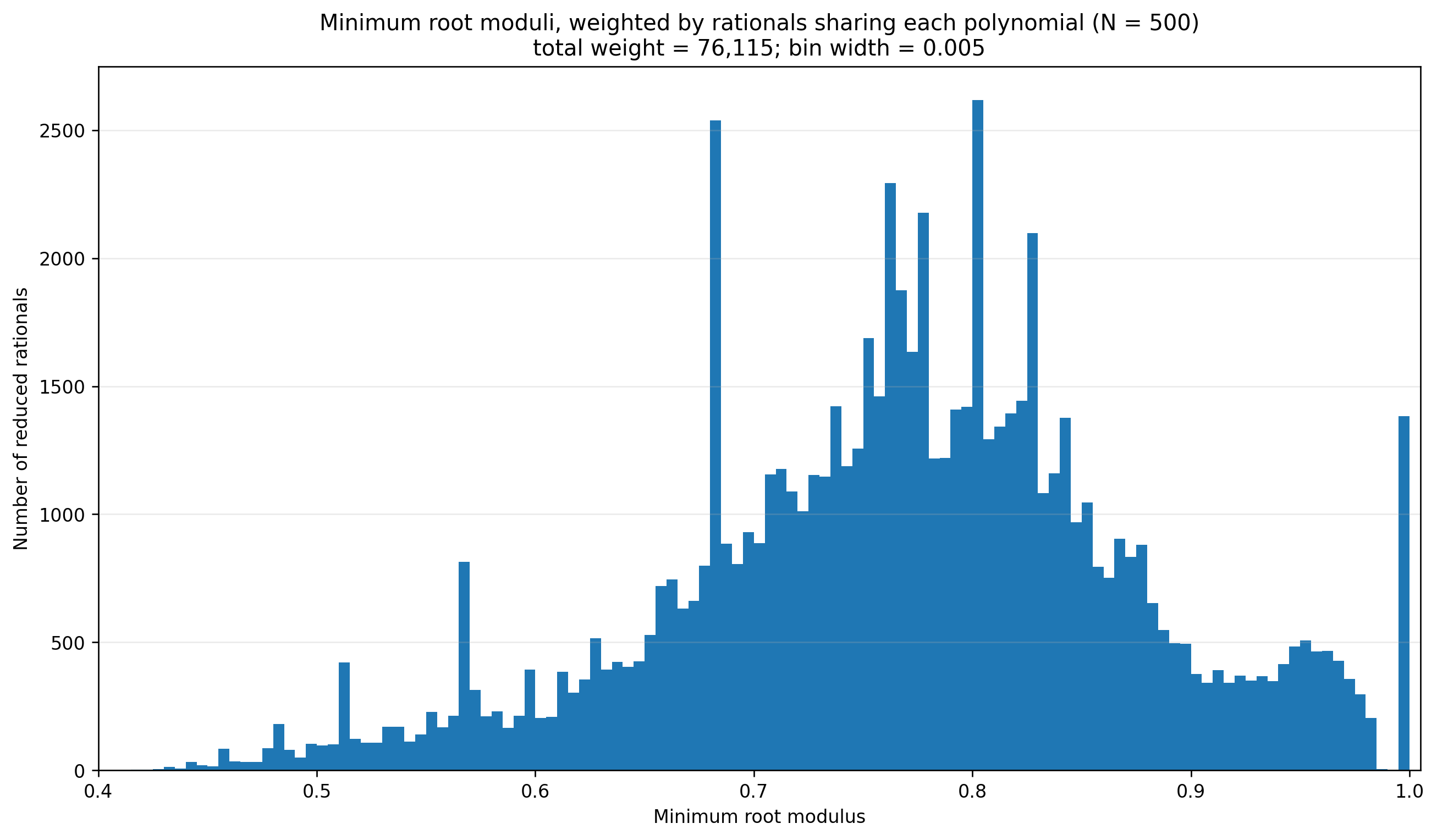}
\caption{\scriptsize Empirical distribution of the minimal pole modulus for reduced rationals
$x\in(1,2)$ of denominator at most $500$, with each rational counted once.}
\label{fig:radii-histogram}
\end{figure}

Each rational is counted once; hence a denominator polynomial
arising from several rationals is counted with that multiplicity.  The resulting
mean and median are both approximately $0.77$.

For comparison, the exact formulas for $[{\rm cotan}(1)]_q$ and $[e]_q$ given
in Section~3 of~\cite{Etingof-qreal} yield the approximate radii
$R_{{\rm cotan}(1)}\approx0.5912689$ and $R_e\approx0.7157974$.  For the more complicated
case of $[\pi]_q$ considered in Example~\ref{PiExample}(b), where noclosed formula is known or expected (even for the usual continuous fraction expansion), numerical estimates based on coefficients through degree
$10^4$ suggest the conjectural value $R_\pi\approx0.57$.

{\bf Acknowledgements.}
Part of this work was carried out during the conference
{\it Representation theory, geometry, and mathematical physics},
held in honor of the 90th birthday of A. A. Kirillov.
We are grateful to the
Simons Center for Geometry and Physics
and to Stony Brook University for hospitality.
We are pleased to thank Sergei Fomin, Sophie Morier-Genoud, Emmanuel Pedon, and Alexander Veselov 
for enlightening discussions.
We used ChatGPT during the development and drafting of the paper 
to explore possible formulations and proof strategies. 
All statements and proofs were independently checked 
and are the sole responsibility of the authors.
This work was partially supported by the NSF grant DMS-2001318.

\end{document}